\DeclareMathAlphabet{\mathpzc}{OT1}{pzc}{m}{it}

\documentclass[headsepline=true]{scrartcl}
\usepackage{amsmath}
\usepackage{amsthm, amssymb, amsfonts,bbm}
\usepackage[top=1.2in,bottom=1.2in,left=1in,right=1in]{geometry}
\usepackage{mathtools}
\usepackage{stmaryrd}
\usepackage[pdfborder={0 0 0}, hypertexnames=false]{hyperref}
\usepackage{amscd}
\usepackage{tensor}
\usepackage{mathrsfs}
\usepackage{arydshln}
\usepackage[dvipsnames]{xcolor}
\usepackage{tikz}
\usetikzlibrary{decorations.markings}
\usetikzlibrary{decorations.pathreplacing}
\usepackage{lscape}
\usepackage{enumerate}
\usepackage{subcaption}
\usepackage[capitalize]{cleveref}
\usepackage{autonum}
\usepackage{multirow}
\usepackage{wrapfig}

\title{Critical points of modular forms}
\author{Jan-Willem van Ittersum%
\thanks{\emph{Email}: \href{mailto:j.w.ittersum@math.uni-koeln.de}{j.w.ittersum@math.uni-koeln.de}, \newline
Max-Planck-Institut f\"ur Mathematik, Vivatsgasse 7, 53111 Bonn, Germany
\newline 
\emph{Current address}: University of
Cologne, Department of Mathematics and Computer Science, Weyertal 86-90, 50931 Cologne, Germany
}
, Berend Ringeling%
\thanks{\emph{Email}: \href{mailto:bjringeling@gmail.com}{bjringeling@gmail.com}, \newline
Department of Mathematics, IMAPP, Radboud University, PO Box 9010, 6500~GL Nijmegen, The Netherlands
}
}

\newcommand{\sltwoz}{\mathrm{SL}_2(\z)}

\DeclareMathOperator{\sgn}{sgn}

\renewcommand{\vec}{\underline}
\renewcommand{\phi}{\varphi}

\newcommand{\z}{\mathbb{Z}}
\newcommand{\q}{\mathbb{Q}}
\renewcommand{\r}{\mathbb{R}}
\renewcommand{\c}{\mathbb{C}}

\newcommand{\dd}{\mathop{}\!\mathrm{d}}
\newcommand{\pdv}[2]{\frac{\partial #1}{\partial #2}}

\theoremstyle{plain} 
\newtheorem{thm}{Theorem}[section]
\newtheorem{lem}[thm]{Lemma} 
\newtheorem{cor}[thm]{Corollary} 
\newtheorem{prop}[thm]{Proposition}

\theoremstyle{definition} 
\newtheorem{defn}[thm]{Definition}
\newtheorem{exmp}[thm]{Example}

\theoremstyle{remark}
\newtheorem{remarkx}{Remark}
\newenvironment{remark}
  {\pushQED{\qed}\remarkx}
  {\popQED\endremarkx}
\newtheorem{question}{Question}

\newcommand{\step}[1]{\mbox{}\newline\noindent{\emph{#1.}}\ \ }

\newcommand{\ii}{\mathrm{i}}
\renewcommand{\Re}{\mathrm{Re}\,}
\renewcommand{\Im}{\mathrm{Im}\,}

\renewcommand{\=}{\: =\: }
\newcommand{\defis}{\: :=\: }
\newcommand{\+}{\,+\,}
\newcommand{\meno}{\,-\,}

\begin{document}
\maketitle
\begin{abstract}
We count the number of critical points of a modular form with real Fourier coefficients 
in a $\gamma$-translate of the standard fundamental domain~$\mathcal{F}$ (with $\gamma\in \sltwoz$). Whereas by the valence formula the (weighted) number of zeros of this modular form in~$\gamma\mathcal{F}$ is a constant only depending on its weight, we give a closed formula for this number of critical points in terms of those zeros of the modular form lying on the boundary of~$\mathcal{F},$ the value of $\gamma^{-1}(\infty)$ and the weight. 
More generally, we indicate what can be said about the number of zeros of a quasimodular form.
\end{abstract}

\section{Introduction}
\paragraph{A valence formula for quasimodular forms}
For a non-zero modular form~$g$ of weight~$k$, the (weighted) number of zeros in a fundamental domain is given by 
\begin{equation}\label{eq:vf} \sum_{\tau \in \gamma\mathcal{F}} \frac{\nu_\tau(g)}{e_\tau} \= \frac{k}{12}\end{equation}
for all $\gamma \in \sltwoz$, where~$\mathcal{F}$ denotes the standard fundamental domain for the action of~$\sltwoz$ on the extended complex upper half plane, $\nu_\tau(g)$ denotes the order of vanishing of~$g$ at~$\tau$ (see \cref{sec:set-up} for the definitions) and~$e_\tau=2$ for an $\sltwoz$-translate of $\ii$, $e_\tau=3$ for an $\sltwoz$-translate of $\rho=-\frac{1}{2}+ \frac{1}{2}\sqrt{3}\,\ii$ and $e_\tau=1$ else, including at the cusp at infinity. Much less is known about the (weighted) number of zeros of derivatives of modular forms, or, more generally, of zeros of quasimodular forms. That is, in this paper we study the value
\[ N_\lambda(f) \defis \sum_{\tau \in \gamma\mathcal{F}} \frac{\nu_\tau(f)}{e_\tau} \qquad\qquad \bigl(\gamma = \left(\begin{smallmatrix} a & b \\ c & d \end{smallmatrix}\right)\in \sltwoz \text{ such that } \lambda=-\tfrac{d}{c}\bigr) \]
for derivatives of modular forms or, more generally, for quasimodular forms~$f$ (the quantity~$N_\lambda(f)$ is well-defined if $f$ is quasimodular). 

As an example, consider the critical points of the modular discriminant~$\Delta$. Note $\Delta'=\Delta E_2$ (with $f'=\frac{1}{2\pi \ii}\frac{\mathrm{d}}{\mathrm{d}\tau}f$), where 
\[E_2(\tau) \= 1 - 24 \sum_{m,r\geq 1} m\,q^{mr}  \qquad\qquad (\tau \in \mathfrak{h}, \text{ the complex upper half plane})\]
is the quasimodular Eisenstein series of weight~$2$ transforming as
\[ (E_2|\gamma)(\tau) \= E_2(\tau) + \frac{12}{2\pi \ii}\frac{c}{c\tau+d} \= E_2(\tau) + \frac{12}{2\pi \ii}\frac{1}{\tau-\lambda(\gamma)}\]
for all $\gamma = \left(\begin{smallmatrix} a & b \\ c & d \end{smallmatrix}\right)\in \sltwoz$ and with $\lambda(\gamma)=-\frac{d}{c}.$
For~$E_2$ the number of zeros in a fundamental domain depends on the choice of this domain. There are infinitely many non-equivalent zeros of~$E_2$; in fact, two zeros are only equivalent if one is a $\z$-translate of the other \cite{BS10}.  Nevertheless, one can still count the number of zeros of~$E_2$ in~$\gamma\mathcal{F}$:
\begin{equation}
\label{eq:E2zeros}N_\lambda(E_2) \=
\begin{cases}
 0 &  |\lambda| \in (\frac{1}{2},\infty] \\
1 &  |\lambda| \in [0,\frac{1}{2}),
\end{cases}\end{equation}
as follows from \cite{IJT14,WY14}.  
Recently, Gun and Oesterl\'e counted the number of critical points of the Eisenstein series~$E_{k}$ for $k>2$ \cite{GO20}
\begin{equation}
\label{eq:DEkzeros}N_\lambda(E_{k}') \=
\begin{cases}
\bigl\lfloor \frac{k+2}{6}\bigr\rfloor+\frac{1}{3}\delta_{k\equiv 2 \, (6)} &  |\lambda|\in (1,\infty] \\[3pt]
\frac{1}{3}\delta_{k\equiv 2 \, (6)} & |\lambda|\in [0,1).
\end{cases}\end{equation}
(Note $E_k$ has a double zero at $\rho=-\frac{1}{2}+ \frac{1}{2}\sqrt{3}\,\ii$ for $k\equiv 2 \mod 6$. Hence, the summand $\frac{1}{3}\delta_{k\equiv 2 \, (6)}$ corresponds to the trivial zero of $E_k'$ at a $\gamma$-translate of $\rho$.)

\paragraph{Critical points of modular forms}
Modular forms admit infinitely many non-equivalent critical points \cite{Seb12}. By counting the number of critical points within a fundamental domain $\gamma\mathcal{F}$ ($\gamma \in \sltwoz$), we refine this statement.
Notice that the only zero of~$\Delta$ in the fundamental domain~$\mathcal{F}$ is at the cusp, whereas the Eisenstein series have all their zeros in~$\mathcal{F}$ on the unit circle \cite{RSD}. Our main theorem expresses the number of critical points of a modular form 
in terms of the number of zeros of this modular form on the boundary of~$\mathcal{F}$. 

Write~$C(g)$ for the number of distinct zeros~$z$ of~$g$ satisfying $|z|=1$ and $-\frac{1}{2}\leq \Re (z)\leq 0$, 
where a zero $z$ is counted with weight $e_z^{-1}$ (i.e., a zero at $\rho=-\frac{1}{2}+ \frac{1}{2}\sqrt{3}\,\ii$ or at~$\ii$ is counted with weight~$\frac{1}{3}$ or~$\frac{1}{2}$ respectively).
Write~$L(g)$ for the number of distinct zeros~$z$ of~$g$ at the cusp or satisfying $\mathrm{Re}(z)=-\frac{1}{2}$ and $|z|> 1$. 
\begin{thm}\label{thm:crit} Let~$g$ be a modular form of weight~$k$ with \emph{real} Fourier coefficients.
Then,
 \begin{align}
N_\lambda(g')\= \frac{k}{12} \+ \begin{cases}
\phantom{-}C(g)\+ \frac{1}{3}\delta_{g(\rho)=0} & |\lambda| \in (1,\infty]\\
-C(g) & |\lambda| \in (\frac{1}{2},1) \\
-C(g)+L(g) & |\lambda| \in [0,\frac{1}{2}).
\end{cases}
\end{align}
\end{thm}
In particular, for any sequence of cuspidal Hecke eigenforms~$g_k$ of weight~$k$ we have
\[ \frac{N_\lambda(g_k')}{N_\lambda(g_k)} \to 1 \quad\quad \text{ as } k\to \infty,\]
as by the holomorphic quantum unique ergodicity theorem the zeros of~$g_k$ are equidistributed in~$\mathcal{F}$ as $k\to\infty$ \cite{HS10}. This leads to the question of whether the zeros of derivatives of Hecke eigenforms are also equidistributed.

Moreover, if~$g$ is a modular form with \emph{all} its zeros in the interior of~$\mathcal{F}$, we find
\[ N_\lambda(g') \= \frac{k}{12} \= N_\lambda(g).\]
(In case all the zeros of~$g$ lie in the interior of~$\mathcal{F}$, then $k\equiv 0 \mod 12$.) Observe that as $\Delta$ has a unique zero at the cusp, we have $C(\Delta)=0$ and $L(\Delta)=1$, by which we recover~\eqref{eq:E2zeros}. Moreover, for modular Eisenstein series we have $C(E_k)=\frac{k}{12}-\frac{1}{3}\delta_{k\equiv 2\,(6)}$ and $L(E_k)=0$, by which we recover~\eqref{eq:DEkzeros}.

\paragraph{Further results}
We aim to generalize these formulae to \emph{all} quasimodular forms, i.e., polynomials in $E_2$ with modular coefficients.
First of all, we show that for any quasimodular form~$f$ the number~$N_\lambda(f)$ only takes finitely many different values if we vary~$\lambda$.
\begin{thm}\label{thm:main1} For a given quasimodular form~$f$, there exist finitely many disjoint intervals~$I$ such that $\mathbb{R}=\cup_{I\in \mathscr{I}} I$, and for each~$I$ there exists a constant~$N_I(f)\in \frac{1}{6}\z$ such that
\[N_\lambda(f) = N_{I}(f) \qquad \text{if }\lambda \in I.\]
\end{thm}
For instance, in \cref{ex:criticalE2} we will see that 
\begin{align}\label{eq:E2'} N_\lambda(E_2') \= \begin{cases}
1 & |\lambda|\in (\frac{1}{v},\frac{1}{2})\cup (v,\infty] \\
0 & |\lambda|\in (0,\frac{1}{v})\cup (\frac{1}{2},v).
\end{cases}\end{align}
for some $v\in (5,6)$, which should be compared to the results in~\cite{CL19}. 

Secondly, we study in more detail the case that $f=f_0+f_1E_2$ for some modular forms~$f_0$ and~$f_1$ of weight~$k$ and $k-2$ respectively and with real Fourier coefficients. For example, 
the first derivative of a modular form can be written in such a way. From now on, assume that $f_0$ and $f_1$ admit no common zeros in the extended upper half plane~${\mathfrak{h}}^*$. 
We give closed formulas for~$N_\lambda(f)$ depending on the behaviour of~$f_1$ at~$\rho$, its zeros on the boundary of~$\mathcal{F}$, and the value of~$f$ at $\ii\infty,\rho$ and these zeros of~$f_1$. That is, let $z_1,\ldots, z_m$ be the zeros of~$f_1$ such that $\Re z_i=-\frac{1}{2}$ and $\Im z_i>\frac{1}{2}\sqrt{3}$, counted with multiplicity and ordered by imaginary part, and let $z_0=\rho$. 
Also, let $\theta_1,\ldots,\theta_n$ be the angles of those zeros of~$f_1$ on the unit circle   satisfying $\frac{2\pi}{3}\geq \theta_1\geq \theta_2\geq \ldots \geq \theta_n\geq \frac{\pi}{2}$ (counted with multiplicity). We introduce the following notation:
\begin{itemize}\itemsep0pt
\item $\widehat{f}(\theta) = e^{\frac{1}{2}k \ii \theta}f(e^{\ii \theta})$
\item $r(f_1)$ denotes the sign of the first non-zero Taylor coefficient in the \emph{natural} Taylor expansion of~$f_1$ around~$\rho$ (see~\eqref{eq:r}); $\nu_\rho(f_1)$ denotes the order of vanishing of $f_1$ at $\rho$. 
\item $s(f)=\sgn a_0(f)$ if~$f$ does not vanish at infinity, and $s(f)=-\sgn a_0(f_1)$ else. Here, $a_0$ denotes the constant term in the Fourier expansion, i.e., $a_0(f)=\lim_{z\to -\frac{1}{2}+\ii \infty} f(z)$. 
\item $w(z_0)=2$ if~$z_0$ equals $\rho,\ii$ or $-\frac{1}{2}+\ii \infty$, and $w(z)=1$ for all other $z\in \mathcal{F}$.
\end{itemize}

\begin{thm}\label{thm:main2} Let $f=f_0+f_1E_2$ be a quasimodular form of weight~$k$ for which~$f_0$ and~$f_1$ are modular forms without common zeros on ${\mathfrak{h}}^*$ and with real Fourier coefficients. Then, there exist constants $N_{(1,\infty]}(f),N_{(\frac{1}{2},1)}(f),N_{[0,\frac{1}{2})}(f)\in \z$ such that 
\[ N_\lambda(f) = \begin{cases}
N_{(1,\infty]}(f) & |\lambda|\in (1,\infty] \\
N_{(\frac{1}{2},1)}(f) & |\lambda| \in (\frac{1}{2},1) \\
N_{[0,\frac{1}{2})}(f) & |\lambda|\in [0,\frac{1}{2}).
\end{cases} \]
Moreover, these constants are uniquely determined by
\begin{align}
\label{eq:Ninf}    N_{(1,\infty]}(f)  &\= \frac{1}{2}\biggl \lfloor \frac{k}{6} \biggr \rfloor \meno  (-1)^{\nu_\rho(f_1)} r(f_1)\sum_{j=1}^n \frac{(-1)^j}{w(e^{\ii \theta_j})}  \,\sgn \widehat{f}(\theta_j)\\
\label{eq:N1/2} N_{(\frac{1}{2},1)}(f) &\= \phantom{\frac{1}{2}} \biggl\lfloor \frac{k}{6} \biggr\rfloor \meno N_{(1,\infty]}(f)  \\
\label{eq:N0} N_{[0,\frac{1}{2})}(f) &\= \phantom{\frac{1}{2}} \biggl\lceil\frac{k}{6}\biggr\rceil  \meno N_{(1,\infty]}(f)
\meno r(f_1)\sum_{j=0}^{m} \frac{(-1)^j}{w(z_j)}\, \sgn f(z_j) - \frac{1}{2} (-1)^{m+1} r(f_1)\, s(f). 
\end{align} 
\end{thm}
This result is proved later as a consequence of \cref{prop:Ninf} and \cref{thm:5.3}, which have similar arguments. However, \cref{thm:5.3} is technically more complicated.

\clearpage
\begin{remark}\mbox{}\\[-18pt]
\begin{enumerate}[{\upshape(i)}]
\itemsep0pt
\item 
Observe that the conditions of the theorem guarantee that the sign function is applied to a non-zero real number, that is, $\widehat{f}(\theta_j),f(z_j)\in \r^*$. 
\item In case~$f_0$ and~$f_1$ do admit common zeros, there always exists a modular form~$g$ such that~$f_0/g$ and~$f_1/g$ are (holomorphic) modular forms without common zeros (see \cref{rk:irreducible1} on page~\pageref{rk:irreducible1}).
\item For quasimodular forms of depth $>1$ (i.e., if~$f$ is a polynomial in~$E_2$ of degree~$>1$ with modular coefficients), the first part of the statement is wrong. This has already been illustrated with the example 
in~\eqref{eq:E2'}; for more details, see \cref{ex:criticalE2}. 
\qedhere
\end{enumerate} 
\end{remark}

\paragraph{An extreme example} 
To illustrate some of the characteristic properties of zeros of quasimodular forms, consider the unique quasimodular form~$f=f_0+f_1E_2$ in the~$7$-dimensional vector space $M_{36}^{\leq 1}$ with $q$-expansion $f=1+O(q^7)$
(as the constant coefficient is $1$, the quasimodular form~$f$ cannot be the derivative of a modular form)\footnote{Following \cite{JP13} one could call this a \emph{gap quasimodular form}. The theta series of an extremal lattice is a gap modular form. Do gap quasimodular forms have a similar interpretation?}. Explicitly, $f$ equals
\[\small 1 \+ 212963830173619200 q^7 \+ 45122255555990230800 q^8 \+ 
 3920264199663225523200 q^9 \+ O(q^{10}). \] \vspace{-25pt}
\begin{figure}[!hbtp] 
\caption{\small Zeros of the unique quasimodular form of weight~$36$, depth~$1$ and of the form $1+O(q^7)$.}\vspace{2pt}\setcounter{figure}{0}
  \begin{subfigure}[b]{0.49\textwidth}
    \includegraphics[width=\textwidth]{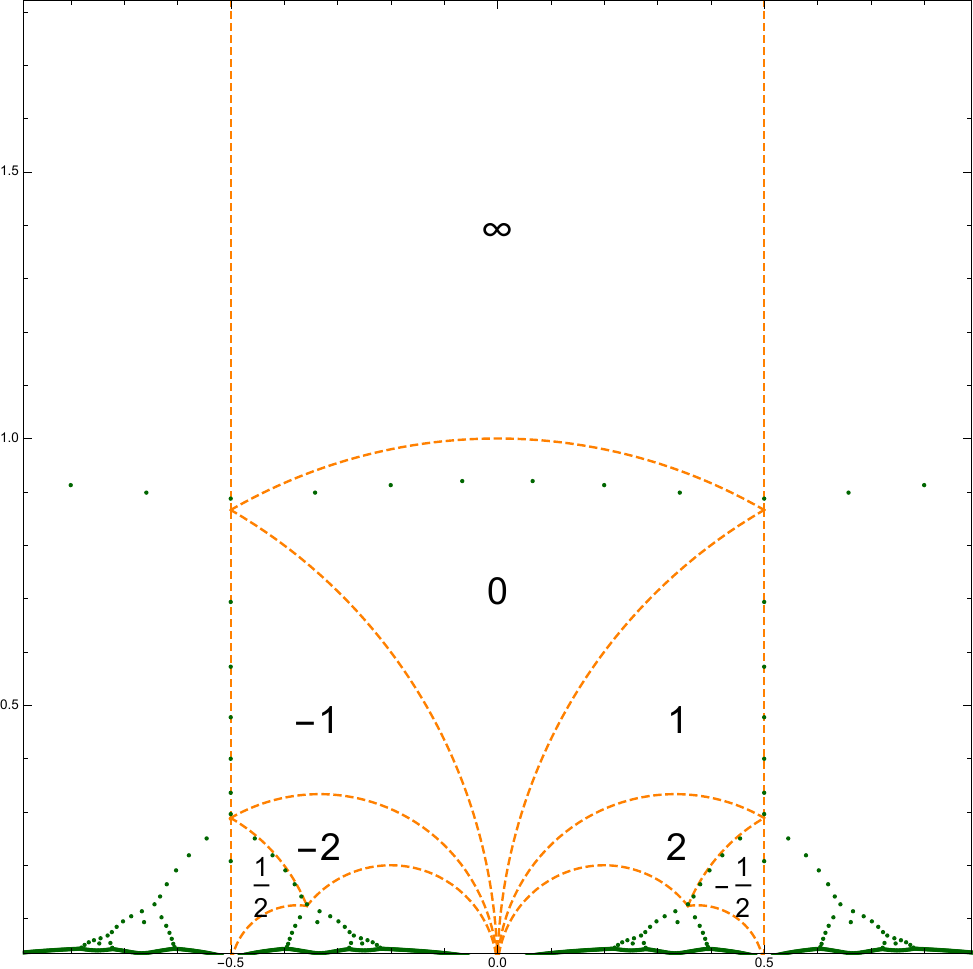}
    \caption{The approximate location of 1000 zeros of~$f$ and fundamental domains~$\gamma\mathcal{F}$ for ${\lambda(\gamma)=0,\pm \frac{1}{2}, \pm {1},\infty}$.}
    \label{fig:zerosa}
  \end{subfigure}
  \hfill
  \begin{subfigure}[b]{0.49\textwidth}
    \includegraphics[width=\textwidth]{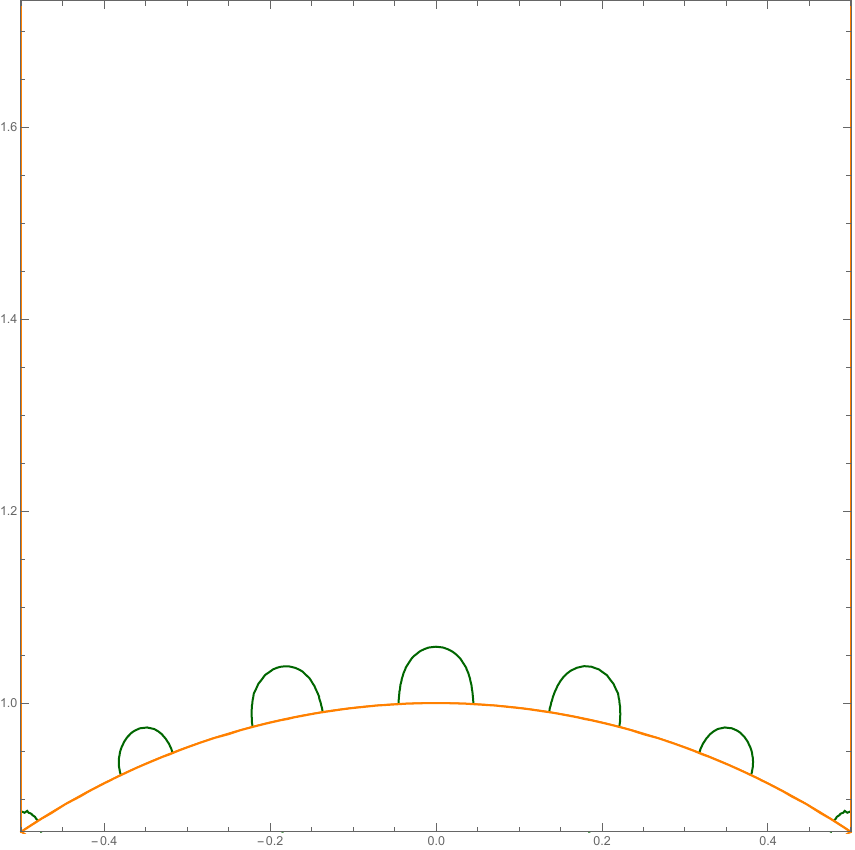}
    \caption{The curves~\eqref{eq:curves} associated to~$f$.\\\phantom{ }\\\phantom{ }}
    \label{fig:zerosb}
  \end{subfigure}
  \label{fig:zeros}
\end{figure}

The zeros of~$f$, depicted in \cref{fig:zerosa}, satisfy
\[ N_{(1,\infty]}=1,\qquad N_{(\frac{1}{2},1)}=5,\qquad N_{[0,\frac{1}{2})}(f)=6.\]
Moreover, in \cref{fig:zerosb}, we depicted the rational curves
\begin{equation}\label{eq:curves}\{\gamma z \mid \gamma \in \sltwoz, f(z)=0\} \cap \mathcal{F}. \end{equation}
By \cref{cor:meromorphic}\eqref{cor:meromorphic(iii)} these curves~\eqref{eq:curves} are given by $\{ z\in \mathcal{F} \mid h(z)\in \q\}$, where the function~$h:\mathfrak{h}\to \c$ is given by
\[ h(\tau) \= \tau + \frac{12}{2\pi i}\frac{f_1(\tau)}{f(\tau)}. \]
In fact,~$h$ is an equivariant function, i.e.,
\[ h(\tau+1) = h(\tau)+1,\qquad h\Bigl(-\frac{1}{\tau}\Bigr) = -\frac{1}{h(\tau)}, \qquad h(-\overline{\tau}) = -\overline{h(\tau)}.\]
For other quasimodular forms of depth~$1$ the corresponding function~$h$ is also equivariant, and these transformation properties are the main ingredients for \cref{thm:main2}.

\paragraph{Extremal quasimodular forms} 
Write $\widetilde{M}_k^{\leq p}$ is the space of holomorphic quasimodular forms of weight~$k$ and depth~$\leq p$. Recall that $\widetilde{M}_4^{\leq 1}=M_4=\c E_4$ and $E_4$ has a unique zero at $\rho$, so that $N_\lambda(E_4)=\frac{1}{3}$. Excluding this modular form, we find the following upper bound.
\begin{cor}\label{cor:upperbound}
For all $f\in \widetilde{M}^{\leq 1}_{k}$ such that $\frac{f}{E_4}\not\in \widetilde{M}^{\leq 1}_{k-4}$, we have
\[N_\lambda(f) \:\leq\: \dim \widetilde{M}^{\leq 1}_{k} \+ \begin{cases} -1 & \lambda\in (\frac{1}{2},\infty] \\
0 & \lambda \in [0,\frac{1}{2}).\end{cases} \]
\end{cor}

Observe that in any vector subspace of~$\c\llbracket q \rrbracket$ of dimension~$m$, there exists an element~$f$ with $v_{\ii \infty}(f)\geq m-1$. Hence, there exists a quasimodular form~$f$ such that (i)~the inequality~\eqref{cor:upperbound} is sharp for $\lambda=\infty$ and (ii)~$f$ admits no zeros in~$\mathcal{F}$ outside infinity.
\begin{cor}\label{cor:extreme} There exists a quasimodular form $f=f_0+f_1E_2\in \widetilde{M}^{\leq 1}_{k}$ such that 
\[N_\infty(f) \=v_{\ii\infty}(f) \= \dim\widetilde{M}^{\leq 1}_{k}-1\]
and all zeros of~$f_1$ in~$\mathcal{F}$ are located on the unit circle. 
\end{cor} 
The existence of a quasimodular form for which $v_{\ii\infty}(f) = \dim\widetilde{M}^{\leq 1}_{k}-1$ was proven by Kaneko and Koiko, who called such a quasimodular form \emph{extremal} \cite{KK06}. It is natural to generalize their question whether $v_{\ii\infty}(f) \leq \dim\widetilde{M}^{\leq p}_{k}-1$ for all $f\in \widetilde{M}_{k}^{\leq p}$ (which has been confirmed for $p\leq 4$ in \cite{Pel20}) to the following one.
 
\begin{question}Let $k,p>0$. Do all $f\in \widetilde{M}_{k}^{\leq p}$ with $\frac{f}{E_4}\not\in \widetilde{M}^{\leq p}_{k-4}$ satisfy
\[ N_\lambda(f) \:\leq\: \dim \widetilde{M}_{k}^{\leq p}\+ \begin{cases} -1 & \lambda\in (\frac{1}{2},\infty] \\
0 & \lambda \in [0,\frac{1}{2})\end{cases} \quad?\]
\end{question}

\paragraph{Contents} We start by recalling some basic properties of quasimodular forms in \cref{sec:set-up}. 
In \cref{sec:h} we discuss equivariant functions~$h$ associated to quasimodular forms of depth~$1$ and of higher depth, which results in the proof of \cref{thm:main1} in Section~\ref{sec:5}.  The proof of \cref{thm:main2} is obtained in Section~\ref{sec:3} (for $\lambda=\infty$) and in Section~\ref{sec:5} (for $\lambda<\infty$). We indicate how \cref{thm:crit} and \cref{cor:upperbound} then follow as corollaries of \cref{thm:main1}. 
Moreover, in all sections we give many additional examples.

\paragraph{Acknowledgements}
We would like to thank Gunther Cornelissen and Wadim Zudilin for inspiring conversations and helpful feedback, as well as Joseph Oesterlé and the anonymous reviewer for many helpful comments on a previous version of this paper. A great part of this research has been carried out in the library of the mathematical institute of Utrecht University (even when both authors were not affiliated with this university any more), as well as during a visit of the second author to the Max-Planck-Institut für Mathematik, for which we would like to thank both institutes.

\section{Set-up: zeros of quasimodular forms}\label{sec:set-up}
\paragraph{Set-up} Fix a holomorphic quasimodular form~$f$ for~$\sltwoz$, of weight~$k$ and depth~$p$, and with real Fourier coefficients, i.e., let $f\in \r[E_2,E_4,E_6]$ of homogenous weight~$k$ and depth~$p$. We write \[ f \=\sum_{j=0}^p f_j \, E_2^j \]
where~$f_j$ is a modular form of weight ${k-2j}$ and $f_p\neq 0$. 

\begin{remark}\label{rk:exp}
For all $\gamma\in\sltwoz$ we have
\begin{align}\label{eq:transfo} (f|_k\gamma)(\tau) \defis (c\tau+d)^{-k} f\Bigl(\frac{a\tau+b}{c\tau+d}\Bigr) \= \sum_{j=0}^p \frac{(\mathfrak{d}^jf)(\tau)}{j!} \Bigl(\frac{1}{2\pi \ii}\frac{c}{c\tau+d}\Bigr)^j,\end{align}
where $\mathfrak{d}$ is the derivation on quasimodular forms uniquely determined by $\mathfrak{d}(E_2) = 12$ and the fact that it annihilates modular forms (see \cite[Section~5.3]{Zag08}), i.e.,
\begin{align}\label{eq:der} \frac{\mathfrak{d}^m(f)}{m!} &\= (12)^m \sum_{j=m}^p \binom{j}{m}  f_j E_2^{j-m} \qquad (m\leq p).\qedhere\end{align}
In fact, one cannot understand the theory of quasimodular forms without recognizing the $\mathfrak{sl}_2$-action on quasimodular forms by the derivation~$D=\frac{1}{2\pi \ii}\pdv{}{\tau} = q \pdv{}{q}$, the weight derivation $W$, which multiplies a quasimodular form with its weight, and the derivation $\mathfrak{d}$, satisfying 
\[ [W, D] = 2 D,\qquad [W, \mathfrak{d}] = -2 \mathfrak{d}, \qquad
[\mathfrak{d}, D] = W. \]
\end{remark}

\begin{remark} Restricting to quasimodular forms with real Fourier coefficients isn't that restrictive, for the following two reasons:
\begin{enumerate}[{\upshape (i)}]
    \item All Hecke eigenforms for $\sltwoz$ have real Fourier coefficients;
    \item Suppose $g$ is a quasimodular with complex, rather than real, Fourier coefficients.  
    Then, $\tilde{g}(\tau) := \overline{g(-\overline{\tau})}$  
    is a quasimodular form which vanishes at~$\tau$ if $g$ vanishes at $-\overline{\tau}$. 
    Hence, $N_\lambda(g)=N_{-\lambda}(\tilde{g})$, $N_\lambda(g \tilde{g})=N_\lambda(g)+N_{-\lambda}(g)$ and $g\tilde{g}$ is a quasimodular form with real Fourier coefficients. Therefore, theorems established for quasimodular forms with real Fourier coefficients, lead to theorems for $N_\lambda(g)+N_{-\lambda}(g)$.
    \qedhere
\end{enumerate}
\end{remark}

\paragraph{The fundamental domain}
Let $\mathfrak{h}=\{z\in \c \mid \mathrm{Im}(z)>0\}$ be the complex upper half plane, $\mathfrak{h}^*=\mathfrak{h}\cup \mathbb{P}^1(\q)$ be the extended upper half plane and
 \[\mathcal{F} \defis \{z \in \mathfrak{h} \mid |z| > 1, -\tfrac{1}{2}\leq \mathrm{Re}(z) <\tfrac{1}{2}\} \,\cup\, \{z \in \mathfrak{h} \mid |z| = 1, -\tfrac{1}{2}
 \leq \mathrm{Re}(z) \leq 0\} \cup \{\ii\infty\}\]
the standard (strict) fundamental domain for the action of $\sltwoz$ on $\mathfrak{h}^*$, where $\ii\infty$ is the point $[1,0]\in \mathbb{P}^1(\q)$ at infinity. Recall that the $\sltwoz$-translates of $\rho=-\frac{1}{2}+\frac{1}{2}\sqrt{3}\, \ii$ and of $\ii$ have a non-trivial stabilizer, i.e., $e_\rho=3, e_\ii=2$ and $e_z=1$ if $z\in \mathfrak{h}^*\backslash\bigl(\sltwoz \rho\cup \sltwoz \ii\bigr)$. 

Moreover, we write $\mathcal{C}, \mathcal{L}$ and $\mathcal{R}$ for the positively oriented circular part, left vertical half-line and right vertical half-line of the boundary $\partial \mathcal{F}$ of $\mathcal{F}$, i.e., $\partial \mathcal{F} = \mathcal{L} \cup \mathcal{C} \cup \mathcal{R} \cup \{\ii\infty\}$ with
\begin{align}\label{eq:C} \mathcal{C} &\= \{ z\in \mathfrak{h} \mid |z|=1, -\tfrac{1}{2}\leq \mathrm{Re}(z)\leq \tfrac{1}{2} \},\\
\label{eq:L} \mathcal{L} &\= \{ z\in \mathfrak{h} \mid |z|\geq 1, \mathrm{Re}(z)= -\tfrac{1}{2}\},\\
\mathcal{R} &\=\{ z\in \mathfrak{h} \mid |z|\geq 1, \mathrm{Re}(z)= \tfrac{1}{2}\}.
\end{align}

\paragraph{Order of vanishing at the cusps}
Note that for a quasimodular form $f$ around $\tau_0=-\frac{d}{c}\in \mathbb{P}^1(\q)$ we have
\[ (c\tau+d)^kf(\tau) \= \sum_{n=1}^\infty a_n(f,\tau,\tau_0) \exp\Bigl(2\pi \ii n\,  \frac{a\tau+b}{c\tau+d}\Bigr),\]
where $a,b\in \z$ are such that $\left(\begin{smallmatrix} a & b \\ c &d \end{smallmatrix}\right)\in \sltwoz$, and with
\[ a_n(f,\tau,\tau_0) \= \sum_{j=0}^p \frac{a_{n,j}}{j!} \Bigl(\frac{-c(c\tau+d)}{2\pi\ii}\Bigr)^j \,\in\, \c[\tau],\]
where~$a_{n,j}$ is the $n$th Fourier coefficient of~$\mathfrak{d}^jf$. 
We define the order of vanishing as follows.

\begin{defn}\label{def:vtau}
 For a quasimodular form~$f$ and $\tau_0 \in \mathfrak{h}$, let $\nu_{\tau_0}(f)$ be the order of vanishing of $f$ at $\tau_0$. If $\tau_0 \in \mathbb{P}^1(\q)$, we let
$\nu_{\tau_0}(f)$ be the minimal value of~$n$ for which $a_n(f,\tau,\tau_0) \in \c[\tau]$ is not the zero polynomial.
\end{defn}
\begin{remark}
Equivalently, for a cusp $\tau_0$ which is not the cusp at infinity we have
\begin{align} \nu_{\tau_0}(f) = \min(\nu_{\ii\infty}(f_0),\ldots,\nu_{\ii\infty}(f_p)). & \qedhere\end{align}
\end{remark}

\paragraph{The counting function}
\begin{defn} Given $\lambda\in \mathbb{P}^1(\q)$, denote by $N_\lambda(f)$ the weighted number of zeros of $f$ in $\gamma\mathcal{F}$, where $\gamma=\left(\begin{smallmatrix} a & b \\ c & d\end{smallmatrix}\right) \in \sltwoz$ and $-\frac{d}{c}=\lambda$, i.e.,
\begin{equation} N_\lambda(f) \= \sum_{\tau \in \gamma\mathcal{F}} \frac{\nu_\tau(f)}{e_\tau},\end{equation}
where $\nu_\tau(f)$ is defined by \cref{def:vtau}. 
\end{defn}
Observe that as $f(\tau+1)=f(\tau)$, the weighted number of zeros in $\left(\begin{smallmatrix} 1 & 1 \\ 0 & 1\end{smallmatrix}\right)\mathcal{F}$ and $\mathcal{F}$ agree. Hence, after fixing a rational number $\lambda=-\frac{d}{c}$ with $c,d$ coprime integers, for all possible choices $a,b\in \z$ such that $ad-bc=1$ the weighted number of zeros in $\left(\begin{smallmatrix} a & b \\ c & d\end{smallmatrix}\right)\mathcal{F}$ agree. 

Without loss of generality, we often restrict to irreducible quasimodular forms:
\begin{defn}
A quasimodular form is \emph{irreducible} if it cannot be written as the product of two quasimodular forms of strictly lower weights. 
\end{defn}

\begin{remark}\label{rk:irreducible1}
If $f$ is a quasimodular form $f=f_0+f_1E_2$ of depth~$1$, then $f$ is irreducible if and only if $f_0$ and $f_1$ have no common zeros. As $f_0$ and $f_1$ are modular, they have a common zero if and only if they have a modular form as common factor. If $z\in \mathcal{F}$ is a common zero, then this modular form is given by
    $E_4$ if $z=\rho$,
    $E_6$ if $z =i$,
    $\Delta$ if $z$ is at the cusp at infinity, and 
   $\Delta(j-j(z))$, where $j$ is the modular $j$-invariant, else. \qedhere
\end{remark}
\begin{remark}\label{rk:irreducible}
Suppose that $f$ is quasimodular with \emph{algebraic} Fourier coefficients. As noted by Gun and Oesterl\'e, if $a\in \mathfrak{h}$ is a zero of~$f$, there exists an irreducible factor $g$ of $f$, unique up to multiplication by a scalar, such that $g$ has a single zero in $a$ \cite[Corollary~2]{GO20}. Hence, if $f$ is irreducible, it has only single zeros. Moreover, if~$f$ has a zero at $\mathrm{i}$ or $\rho$ (or one of their $\sltwoz$-translates), then it has $E_6$ or $E_4$ respectively as one of its factors. In particular, if~$f$ is an irreducible quasimodular form with algebraic Fourier coefficients, then
\begin{align} N_\lambda(f) = \sum_{\tau \in \gamma \mathcal{F}} \nu_\tau(f) \:\in\: \z_{\geq 0}\, 
\end{align}
if $f$ is not a modular form.
\end{remark}

\paragraph{Local behaviour of modular forms around~$\rho$}
Recall $\rho=-\frac{1}{2}+\frac{1}{2}\sqrt{3}\,\ii$. Let $g$ be a modular form of weight~$k$ \emph{with real Fourier coefficients}. Note that the mapping $w\mapsto \frac{\rho-\overline{\rho}w}{1-w}$ maps the unit disc to~$\mathfrak{h}$. 
Then, the natural Taylor expansion of $g$ on $\mathfrak{h}$ (see \cite[Proposition~17]{Zag08}) around $\tau=\rho$ is given by
\[ (1-w)^{-k} \,g\Bigl(\frac{\rho-\overline{\rho}w}{1-w}\Bigr) \= \sum_{n=\nu_\rho(g)}^\infty b_n(g) \, w^n \qquad (|w|<1),\]
for some $\nu_\rho(g)\geq 0$ and coefficients $b_n(g)\in \mathbb{R}$ with $b_{\nu_\rho(g)}\neq 0$. (This Taylor expansion is natural as the image of $w \mapsto \frac{\rho-\overline{\rho}w}{1-w}$ for $|w|<1$ equals the full domain~$\mathfrak{h}$ on which $g$ is holomorphic.) Alternatively, $g$ admits an ordinary Taylor expansion
${g(z) = \sum_{n=\nu_\rho(g)}^\infty c_n(g) \, (2\pi\ii)^n(z-\rho)^n}$ (with $|z|$ sufficiently small, and for the same value of $\nu_\rho(g)$) and for some coefficients $c_n(g)\in \c$ with $c_{\nu_\rho(g)}(g)\neq 0$. Let 
\begin{equation}\label{eq:r} r(g) \defis \sgn b_{\nu_\rho(g)}(g).\end{equation}
In the sequel, we need the following relation between $r(g)$ and the limiting behaviour of $g$ on the boundary of $\mathcal{F}$.
\begin{lem}\label{lem:r} Let $g$ be a modular form of weight~$k$ with real Fourier coefficients. Then, for all $t\in \r_{>0}$ and $0<\theta<\pi$ the values of $g(\rho+\ii t)$ and $e^{k\ii\theta/2}g(e^{\ii \theta})$ are real. 
Moreover, 
\[ \lim_{t\downarrow 0} \sgn(g(\rho+\ii t)) \= r(g) \= (-1)^{\nu_\rho(g)}\sgn(c_{\nu_\rho(g)}) \=  (-1)^{\nu_\rho(g)} \lim_{\theta\uparrow2\pi/3} \sgn(e^{k\ii\theta/2}g(e^{\ii \theta})), \]
where $r(g)$ is defined by~\eqref{eq:r} and $c_n$ are the Taylor coefficients of $g$ around $\rho$ as above.
\end{lem}
\begin{proof}
The fact that $g(\rho+\ii t)$ is real for real $t$, follows directly from the assumption that the Fourier coefficients of $g$ are real. Moreover, this assumption implies that
\begin{align}\label{eq:imghat}
    \overline{e^{k\ii\theta/2}g(e^{\ii \theta})} \= e^{-\ii k\theta/2}g\Bigl(\frac{-1}{e^{\ii \theta}}\Bigr)
    \= e^{-\ii k\theta/2} e^{\ii k\theta}  g(e^{\ii \theta})
\end{align}
Hence, $\Im e^{k\ii\theta/2}g(e^{\ii \theta}) =0$. 

Now, note $g(\rho+\ii t)=g(\frac{\rho-\overline{\rho}w}{1-w})$ for $w=\frac{t}{\sqrt{3}+t}.$ Hence, 
\[ \lim_{t\downarrow 0} \sgn(g(\rho+\ii t)) \= \lim_{w\downarrow 0} \sgn g\Bigl(\frac{\rho-\overline{\rho}w}{1-w}\Bigr) \= \sgn(b_{\nu_\rho(g)}).\]
Also, $g(\rho+\ii t) = \sum_{n=\nu_\rho(g)}^\infty c_n(g) (-2\pi t)^n$, hence, $\sgn(b_{\nu_\rho(g)})=(-1)^{\nu_\rho(g)}\sgn(c_{\nu_\rho(g)}).$ Finally,
for the last equality, we observe that by the valence formula, we know that $g$ has order ${\nu_\rho(g)}=3\ell+\delta$ at $\rho$ for some non-negative integer~$\ell$. Here, $\delta\in \{0,1,2\}$ is the reduced value of $k\mod 3$. In particular, in all cases we find that
\[e^{k\ii\theta/2}g(e^{\ii\theta}) \sim (-2\pi)^{{\nu_\rho(g)}} (\theta -\tfrac{2 \pi }{3})^{{\nu_\rho(g)}} c_{{\nu_\rho(g)}}\]
as $\theta\uparrow 2\pi/3$. 
\end{proof}

\section{Zeros in the standard fundamental domain \texorpdfstring{$(\lambda=\infty)$}{}}\label{sec:3}
Let $f$ be a quasimodular form. In order to compute $N_\infty(f)$, we compute the contour integral of the logarithmic derivative of $f$ over the boundary of $\mathcal{F}$ (suitably adapted with small circular arcs, if $f$ has zeros on this boundary). For simplicity of exposition, assume $f$ has no zeros on the circular part of the boundary~$\mathcal{C}$. Then, by a standard argument
\begin{equation}\label{eq:Ninfty1}
N_\infty(f) \= \frac{1}{2 \pi \ii}\int_{\mathcal{C}} \frac{f'(z)}{f(z)} \,\dd z \= -\frac{1}{2\pi \ii} \int_{\frac{\pi}{3}}^{\frac{2 \pi}{3}} \frac{\dd}{\dd \theta} \log(f(e^{\ii\theta}))  \, \dd \theta.
\end{equation}
If $g$ is quasimodular of weight $k$, we define $\widehat{g}:[\frac{\pi}{3}, \frac{2 \pi}{3}]\to \c$ by
\begin{equation}
    \widehat{g}(\theta) = e^{k\ii\theta/2}g(e^{\ii\theta}).
\end{equation}
We express $N_\infty(f)$ in terms of $\widehat{f}$, as follows
\begin{equation}
  N_\infty(f) \=   -\frac{1}{2 \pi \ii} \int_{\frac{\pi}{3}}^{\frac{2 \pi}{3}} \frac{\dd}{\dd \theta} \log(e^{-k\ii\theta/2}\widehat{f}(\theta)) \, \dd \theta 
  \= \frac{k}{12} - \frac{1}{2 \pi \ii} \int_{\frac{\pi}{3}}^{\frac{2 \pi}{3}} \frac{\widehat{f}'(\theta)}{\widehat{f}(\theta)} \, \dd \theta.
\end{equation}
Since $N_\infty(f)$ is real-valued, we find
\begin{equation}
N_\infty(f) \=   \frac{k}{12} - \frac{1}{2 \pi} \Im\!\left(\int_{\frac{\pi}{3}}^{\frac{2 \pi}{3}} \frac{\widehat{f}'(\theta)}{\widehat{f}(\theta)} \, \dd \theta \right).
\end{equation}
We have the following interpretation for the latter integral. 
Write $\widehat{f}(\theta) = r(\theta) e^{2\pi \ii \alpha(\theta)}$, where $r$ and $\alpha$ are real-valued continuous functions, i.e., $r$ is the radius of $\widehat{f}$ and $\alpha$ is called the \emph{continuous argument} of~$\widehat{f}$. Recall that by assumption $f$ has no zeros on $\mathcal{C}$, so $r(\theta)>0$ for all~$\theta$.
Then,
\begin{align}\label{eq:Ninfty2} N_\infty(f) \= \frac{k}{12} - \Big(\alpha\Bigl(\frac{2\pi}{3}\Bigr)-\alpha\Bigl(\frac{\pi}{3}\Bigr)\Big).\end{align}
In order to compute \emph{the variation of the argument} $\alpha(\frac{2\pi}{3})-\alpha(\frac{\pi}{3})$, we first determine all $\theta\in [\pi/3,2\pi/3]$ for which $\alpha(\theta) \in \frac{1}{2}\z$, or equivalently, for which $\Im (\widehat{f})=0$. By making use of the assumption that our quasimodular form~$f$ has \emph{real} Fourier coefficients, we obtain:
\begin{lem}\label{lem:Imhatf}
We have
\begin{align}
    \Im (\widehat{f}) &\=  \frac{\ii}{2}
    \sum_{m\geq 1}\frac{1}{(2\pi \ii)^m m!} \widehat{\mathfrak{d}^mf} .
\end{align}
\end{lem}
\begin{proof}
First, assume~$g$ is a modular form (rather than a \emph{quasi}modular form) of homogeneous weight~$k$. 
Then, using that~$g$ has real Fourier coefficients, have
\begin{align}\label{eq:imghat}
    \overline{\widehat{g}(\theta)} = e^{-\ii k\theta/2}g\Bigl(\frac{-1}{e^{\ii \theta}}\Bigr)
    = e^{-\ii k\theta/2} e^{\ii k\theta}  g(e^{\ii \theta})
\end{align}
Hence, $\Im\widehat{g}=0$. Similarly, for $\widehat{E}_2$ one has
\begin{align}\label{eq:imE2hat}
    \overline{\widehat{E}_2(\theta)} = e^{-\ii \theta}E_2\Bigl(\frac{-1}{e^{\ii \theta}}\Bigr)
    = e^{-\ii \theta} \bigl(e^{2\ii\theta}  E_2(e^{\ii \theta})+\tfrac{12}{2\pi i}e^{\ii\theta}\bigr)
    = \widehat{E}_2(\theta) + \frac{12}{2\pi \ii}
\end{align}
Hence, 
\[ \Im \widehat{E}_2^j = \frac{\ii}{2}\sum_{m=1}^{j} \binom{j}{m}\Bigl(\frac{12}{2\pi \ii }\Bigr)^{m} \widehat{E}_2^{j-m}  \]
Applying this to the expansion
$f=\sum_{j\geq 0} f_j \, E_2^j$ and using the expansion as in \cref{rk:exp}, we find
\begin{align}
    \Im (\widehat{f}) &\= \frac{\ii}{2}
    \sum_{m\geq 1}\sum_{j\geq m} \binom{j}{m}\Bigl(\frac{12}{2\pi \ii}\Bigr)^{m} \, \widehat f_j \widehat{E}_2^{j-m} 
    \=   \frac{\ii}{2}
    \sum_{m=1}^{j}\frac{1}{(2\pi \ii)^m m!} \widehat{\mathfrak{d}^mf}. \qedhere
\end{align}
\end{proof}

\paragraph{Depth 1}
We now restrict to irreducible quasimodular forms of depth~$1$, i.e., $f=f_0+E_2f_1$. Then, by the previous lemma we have $\Im(\widehat{f}) = \frac{3}{\pi} \widehat{f_1}$. We write
\[\frac{2\pi}{3}\geq \theta_1 > \ldots >\theta_n > \frac{\pi}{3}\]
for the zeros of $\theta\mapsto f_1(e^{\ii \theta})$ on $(\frac{\pi}{3},\frac{2\pi}{3}]$, counted with multiplicity.
Recall that ${\nu_\rho(f_1)}$ denotes the order of vanishing of $f_1$ at $\rho$. 
Then, Equation~\eqref{eq:Ninf} in \cref{thm:main2} will follow from the following lemma and proposition.
\begin{lem} The mapping $\iota:\theta_j\mapsto \theta_{n-j-{\nu_\rho(f_1)}+1}$ defines an involution on $\{\theta_j \mid \theta_j\neq \frac{2\pi}{3}\}$ such that \[(-1)^j\sgn(\widehat{f}(\theta_j)) = (-1)^{n-j-{\nu_\rho(f_1)}+1} \sgn(\widehat{f}(\theta_{n-j-{\nu_\rho(f_1)}+1})).\]
\end{lem}
\begin{proof}
Note that if $\theta_j$ is the angle of an element on the unit disk for which $f_1$ has a zero, then $\pi-\theta_j$ also is such an angle. Leaving out the~${\nu_\rho(f_1)}$ angles $\theta_1=\ldots=\theta_{\nu_\rho(f_1)}=\frac{2\pi}{3}$, we see that $\iota$ is a well-defined involution.

As $f_1(e^{\ii\theta_j})=0$, we obtain
\begin{align} \widehat{f}(\theta_{n-j-{\nu_\rho(f_1)}+1}) &\= e^{\frac{1}{2}\ii k(\pi-\theta_j)}\, f\Bigl(-\frac{1}{e^{\ii \theta_{j}}}\Bigr) 
\= e^{\frac{1}{2}\ii k(\pi-\theta_j)}\, e^{k\ii\theta_{j}}\, f(e^{\ii \theta_{j}}) 
\= (-1)^{k/2} \,\widehat{f}(\theta_j).
\end{align}
We finish the proof by showing that $\frac{k}{2}\equiv n-{\nu_\rho(f_1)}+1 \mod 2$. Namely, $n-{\nu_\rho(f_1)}$ is odd precisely if $f_1$ admits a zero of odd order at $\ii$, or, equivalently, if $k-2\equiv 2,6$ or $10 \mod 12.$ We can exclude the case where $k\equiv 4 \, (6)$. Namely, then both $f_0$ and $f_1$ are divisible by $E_4$, contradicting the irreducibility of $f$. Hence,  $n-{\nu_\rho(f_1)}$ is odd if $k\equiv 0,8 \mod 12$ and even if $k\equiv 2,6 \mod 12$ as desired. 
\end{proof}

\begin{prop}\label{prop:Ninf}
For an irreducible quasimodular form~$f$ of weight~$k$ and depth 1, we have
\begin{align}\label{eq:zeros1}
N_\infty(f)  &\= \frac{1}{2}\left \lfloor \frac{k}{6} \right \rfloor - \frac{(-1)^{\nu_\rho(f_1)}r(f_1)}{2}\sum_{j} (-1)^j \sgn(\widehat{f}(\theta_j)).
\end{align}
\end{prop}
\begin{proof}
The idea of the proof is to determine the value $\alpha(\frac{2\pi}{3})-\alpha(\frac{\pi}{3})$ in~\eqref{eq:Ninfty2}. 
Denote by $A(\theta)$ the argument of $\widehat{f}$, i.e., the unique value in $(-\frac{1}{2},\frac{1}{2}]$ such that $\alpha(\theta)\equiv A(\theta)\mod 1$. 

As $\alpha$ is real analytic, this value can uniquely be determined by knowing $A(\frac{\pi}{3}), A(\frac{2\pi}{3})$ and all the values of $\theta$ for which $A(\theta)\in \{0,\frac{1}{2}\}$. For example, if $0<A(\frac{2\pi}{3})<\frac{1}{2}$, and $A(\theta_1)=\frac{1}{2}$, whereas $A(\theta_2) =0$, then for $\frac{2\pi} 3>\theta\geq  \theta_2$, $\alpha$ increases by $1-A(\frac{2\pi} 3).$ Observe that $A(\theta)\in \{0,\frac{1}{2}\}$ precisely if $\Im f(e^{\ii \theta})=0$, or equivalently, $f_1(e^{\ii \theta})=0$. 

Now, in order to compute the value of $\alpha(\frac{2\pi}{3})-\alpha(\frac{\pi}{3})$, first assume that all zeros of $\theta\mapsto f_1(e^{\ii \theta})$ on $(\frac{\pi}{3},\frac{2 \pi}{3}]$ are simple and satisfy $\theta\in \{\frac\pi 3,\frac\pi 2\}$.
Whether $f_1(e^{\ii\theta})=0$ for such $\theta\in \{\frac\pi 3,\frac\pi 2\}$ (or, equivalently, $A(\frac{2\pi}{3})\in \{0,\frac{1}{2}\}$) is determined by the value of $k$ modulo $12$, see below. Note that as $f$ is \emph{irreducible}, we have $k\not\equiv 4 \, (6)$.
\[ 
\begin{array}{r c c}
k \mod 12  & A(\frac{2\pi}{3})\in \{0,\frac{1}{2}\} &  A(\frac\pi 2)\in \{0,\frac{1}{2}\}    \\\hline
0 & \checkmark & \checkmark\\
2 & \mathbf{x} & \mathbf{x} \\
6 & \checkmark &\mathbf{x} \\
8 & \mathbf{x} & \checkmark 
\end{array}
\]
Temporarily, denote by $\phi_i$ the elements of $\{2\pi/3, \pi/2\}$ for which $\theta\mapsto f_1(e^{\ii \theta})$ admits a zero, and such that $\phi_1\geq \phi_2$.
As $f$ is irreducible, we have $\widehat{f}(\phi_i)\neq 0$, so that  $\sgn(\widehat{f}(\phi_i))$ is well-defined. The sign being positive (or negative) corresponds to $\alpha(\phi_i)\equiv 0 \mod 1$ (or $\frac{1}{2} \mod 1$ respectively). 
By \cref{lem:r} we have $(-1)^{{\nu_\rho(f_1)}}r(f_1) = \lim_{\theta\uparrow2\pi/3} \sgn(e^{k\ii\theta/2}f_1(e^{\ii \theta}))$ with ${\nu_\rho(f_1)}$ the order of vanishing of $f_1$ at $\rho$. Hence, a case-by-case analysis using the symmetry
$\Im \widehat{f}(\theta) = (-1)^{k/2+1} \Im \widehat{f}(\pi-\theta)$ shows
\[A\Bigl(\frac{2\pi}{3}\Bigr)-A\Bigl(\frac{\pi}{3}\Bigr) \meno \frac{(-1)^{\nu_\rho(f_1)}r(f_1)}{2}\sum_{j} (-1)^j \sgn(\widehat{f}(\phi_j)) \= 
\begin{cases}
0 & k\equiv 0 \, (6) \\
\frac{1}{6} & k\equiv 2 \, (6).
\end{cases}\]

Now, in the general case, note that the contribution to the variation of the argument on each interval $[\theta_j, \theta_{j+1}]$ is \[
\frac{(-1)^{\nu_\rho(f_1)}r(f_1)}{4} \left((-1)^j \sgn(\widehat{f}(\theta_j)) + (-1)^{j+1} \sgn(\widehat{f}(\theta_{j+1}))\right).\]
Adding these contributions with special care at the boundary cases as above leads to the result
\[\alpha\Bigl(\frac{2\pi}{3}\Bigr)-\alpha\Bigl(\frac{\pi}{3}\Bigr) \meno \frac{(-1)^{\nu_\rho(f_1)}r(f_1)}{2}\sum_{j} (-1)^j \sgn(\widehat{f}(\theta_j)) \= 
\begin{cases}
0 & k\equiv 0 \, (6) \\
\frac{1}{6} & k\equiv 2 \, (6).
\end{cases}\]
By Equation~\eqref{eq:Ninfty2} the result follows.
\end{proof}

\begin{remark}
For a mixed modular form~$F=\sum_{j=0}^p f_j$ with $f_j$ of weight $k-2j$, we analogously find
\begin{align}
    \Im (\widehat{F}) &= \sum_{j\geq 1}\widehat{f_j}(\theta)\sin(j\theta). 
\end{align}
From this we similarly deduce that for a mixed modular form~$F=f_0+f_j$ (with $f_j$ of weight $k-2j$) we have
\begin{align}\label{eq:zerosmixed}
N_\infty(F) &\= \frac{1}{2}\left \lfloor \frac{k}{6} \right \rfloor - \frac{(-1)^{\nu_\rho(f_1)}r(f_j)}{2}\sum_{i} (-1)^i \sgn(\widehat{F}(\theta_i)),
\end{align}
where, accordingly, the $\theta_i$ are the zeros of $\theta\mapsto f_j(e^{\ii\theta})$. 
\end{remark}

\paragraph{Examples in depth 1}
\begin{exmp}\label{E2Exp}
Consider $f = E_2$. In this case, $f_0 \equiv 0$ and $f_1 \equiv 1$. As $f_1$ has no zeros on the arc, application of Proposition~\ref{prop:Ninf} gives
\begin{align}
    N_\infty(E_2) = \frac{1}{2}\left \lfloor \frac{2}{6} \right \rfloor = 0.
\end{align}
Hence, $E_2$ has no zeros in the standard fundamental domain---a result which was discovered and proven in \cite[Proposition~4.2]{BS10} by different means.
\end{exmp}

\begin{exmp}\label{ex:intro}
We now return to the example in the introduction, i.e., let $f$ be the unique quasimodular form~$f=f_0+f_1E_2$ in the~$7$-dimensional vector space $M_{36}^{\leq 1}$ with $q$-expansion $f=1+O(q^7)$.
In order to apply Theorem~\ref{thm:main2}, 
we compute $f_0$ and $f_1$ explicitly:
\[
     f_0 \= \frac{43976643}{108264772}{\Delta^3}\left(j^3-\frac{28903981960}{14658881}j^2\+\frac{9706007861928}{14658881}j\+\frac{396402626858112}{14658881}\right) 
\]
and
\[
    f_1 \= \frac{64288129}{108264772}{E_4E_6}{\Delta^2}\left(j^2 - \frac{2225338584}{1737517}j + \frac{373036607496}{1737517}\right), 
\]
where $j$ is the modular $j$-invariant, given by $j=1728\frac{E_4^3}{E_4^3-E_6^2}$. 
We find that $f_1$ has zeros at $\ii$ and $\rho$, coming from the factors $E_4E_6$. Moreover, the roots of the degree~$2$ polynomial in the $j$-invariant are given by $j(\tau_1) \approx 198.3495\ldots$ and $j(\tau_2) \approx 1082.4083\ldots$. Recall $j(\mathcal{L}) = (-\infty,0]$ and $j(\mathcal{C}) = [0,1728]$, where $\mathcal{L}$ and $\mathcal{C}$ are the left vertical and circular boundary of the fundamental domain as in~\eqref{eq:C} and~\eqref{eq:L}. Therefore, the zeros of $f_1$ in $\mathcal{F}$ are all located on $\mathcal{C}$.
Similarly, the zeros of $f_0$ are $\tau_3, \tau_4, \tau_5$, where $j(\tau_3) \approx -36.7451\ldots$, $j(\tau_4) \approx 482.1402\ldots$ and $j(\tau_5) \approx 1526.3776\ldots$, indicating that $\tau_4$ and $\tau_5$ lie on~$\mathcal{C}$ (and $\tau_3$ lies on~$\mathcal{L}$).

As $\Delta(\rho) < 0$ and $j(\rho) = 0$, we have that
\[ 
f(\rho) \= \widehat{f}\Bigl(\frac{2\pi}{3}\Bigr) \= \widehat{f_0}\Bigl(\frac{2\pi}{3}\Bigr) \= f_0(\rho) \:<\: 0.
\]
From the location of the zeros of $f_0$ on $\mathcal{C}$, we conclude (writing $\tau_j = e^{\ii \theta_j}$ for $\frac{\pi}{2} \leq \theta_j \leq \frac{2 \pi}{3}$ if $\tau_j$ is located on $\mathcal{C}$)
\begin{align}
    \widehat{f}(\theta_1) &\= \widehat{f_0}(\theta_1) \:<\: 0, \\
    \widehat{f}(\theta_2) &\= \widehat{f_0}(\theta_2) \:>\: 0, \\
    \widehat{f}(\tfrac{\pi}{2}) &\= \widehat{f_0}(\tfrac{\pi}{2}) \:<\: 0 .
\end{align}
Further, $f_1(-\frac{1}{2} + \ii \infty) > 0$ and $f_1$ does not change sign on $\mathcal{L}$, as it has no zeros there. Therefore, 
$r(f_1)=1$ and $(-1)^{\nu_\rho(f_1)}=-1$.
We now apply \cref{prop:Ninf} (in the form of Equation~\eqref{eq:Ninf}): 
\begin{align}
    N_{(1, \infty]}(f) \= \frac{1}{2} \left \lfloor \frac{36}{6} \right \rfloor \+ \Bigl(\frac{-1}{2}\cdot -1 + 1 \cdot -1 + -1 \cdot 1 + \frac{1}{2} \cdot -1 \Bigr) \= 1.
\end{align}
We come back to this example in \cref{ex:intro2}.
\end{exmp}

\begin{exmp}
For $k = 6n$ $(n = 1, 2, 3, \dots)$, we consider the Kaneko--Zagier differential equation 
\begin{equation}
    f''(\tau) - \frac{k}{6}E_2(\tau) f'(\tau) + \frac{k(k-1)}{12}E_2'(\tau)f(\tau) = 0.
\end{equation}
In \cite[Theorem 2.1]{KK06}, it was shown that a solution to this equation is given by an \emph{extremal} quasimodular form $g$ of depth $1$ and weight $k$, i.e.

\begin{equation}
    g(\tau) \= cq^{m-1} + \mathcal{O}(q^m),
\end{equation}
where $c \neq 0$ and $m$ is the dimension of the space of weight $k$ forms of depth $1$ (i.e., $m=n+1$).
Clearly, $g$ has a zero at $i \infty$ of order $m-1=\frac{k}{6}$. From \cref{prop:Ninf} we learn that $g$ has no other zeros in $\mathcal{F}$ (see also \cref{cor:extreme}).
\end{exmp}

\paragraph{Examples in higher depth}
In depth~$1$ we have seen that the number of zeros of a quasimodular form $f=f_0+f_1E_2$ only depends on the sign of $\widehat{f_0}$ in the zeros of $f_1$ on the arc $\mathcal{C}$. The next example shows that this is not anymore the case for higher depth.
\begin{exmp}
Consider the following quasimodular form of weight~$4$ and depth~$2$ for a real parameter~$t$
\begin{equation}
    f_t = E_4 - t\,E_2^2.
\end{equation}
We are interested in the value of $N_\infty(f_t)$.
By \cref{lem:Imhatf} we have
\begin{align}
    \Im(\widehat{f_t}) &\= 
    -t \frac{6}{\pi}\Bigl(\widehat{E}_2 + \frac{3}{\pi \ii} \Bigr). 
\end{align}
It can be seen that $\Im(\widehat{f_t})$ only vanishes once, at $\theta_0 \, = \frac{\pi}{2}$. Hence, 
$$\widehat{f_t}(\theta_0) \= \widehat{E}_4(\theta_0) + t \frac{9}{\pi^2}.$$

Now, first assume $t<0$. 
As $\widehat{E}_4 < 0$ on $(\frac{\pi}{3},\frac{2\pi}{3})$, we have $\widehat{f_t}(\theta_0) < 0$.
Also $\widehat{f}_t(\pi/3) = r e^{2 \pi \ii/3}$ for some positive $r$ and $\widehat{f}_t(2 \pi /3) = r e^{4 \pi i /3}$.
This means that $\widehat{f}_t(\theta)$ with $\theta \in (\frac{\pi}{3},\frac{2\pi}{3})$ moves from $re^{2 \pi \ii/3}$ to $re^{4 \pi\ii /3}$, crossing the (negative) real axis exactly once. Hence, the variation of the argument $\alpha(\frac{2\pi}{3})-\alpha(\frac{\pi}{3}) = \frac{2 \pi}{3}$.
Therefore, $$N_\infty(f_t) = \frac{4}{12} - \frac{1}{2 \pi} \frac{2 \pi}{3} = 0,$$ for $t < 0$.

For $t > 0$, we have two cases.
Let
$$t_1 \defis - \frac{\pi^2}{9} \widehat{E}_4(\theta_0) \approx 1.596.. .$$
Now assume $0 < t < t_1$. Since $t>0$, we have
$\widehat{f}_t(\pi/3) = s e^{5 \pi \ii/3 }$ for some positive $s$ and $\widehat{f}_t(2 \pi /3 ) = s e^{\pi \ii/3}$. Since $t < t_1$, we still have $\widehat{f}_t(\theta_0) < 0$.
Hence the variation of the argument is now $-\frac{4 \pi}{3}$. Therefore, 
$$N_\infty(f_t) = \frac{4}{12} - \frac{1}{2 \pi} \cdot \frac{-4 \pi}{3} = 1.$$
For the case $t > t_1$, we have that $\widehat{f_t}(\theta_0) > 0$. So the variation of the argument equals $\frac{2 \pi}{3}$ in that case, and $N_\infty(f_t)=0$.

We conclude that
$$
N_\infty(f_t) =
\begin{cases}
0 &\text{if } t < 0 \text{ or } t> t_1\\
1 &\text{if } 0 < t < t_1\,. \\
\end{cases}
$$
\end{exmp}

\section{Vector-valued equivariant forms}\label{sec:h}
By the quasimodular transformation equation~\eqref{eq:transfo}
\[ (f|_k\gamma)(\tau)  \= \sum_{j=0}^p \frac{(\mathfrak{d}^jf)(\tau)}{j!} \Bigl(\frac{1}{2\pi \ii}\frac{1}{\tau-\lambda}\Bigr)^j,\]
where $\lambda=-\frac{d}{c}$ for $\gamma=\left(\begin{smallmatrix} a & b \\ c & d \end{smallmatrix}\right) \in \sltwoz$. We study the solutions $h:\mathfrak{h}\to \c$ of 
\begin{align}\label{eq:defh} 0 \= \sum_{j=0}^p \frac{(\mathfrak{d}^jf)(\tau)}{j!} \Bigl(\frac{1}{2\pi \ii}\frac{1}{\tau-h(\tau)}\Bigr)^j.\end{align}
The main property of the solutions~$h$ is that $f$ has a zero at $\gamma\tau$ with $\tau\in \mathcal{F}$ if and only if there is a solution satisfying $h(\tau)=\lambda$. 
Another property 
is that if $h_1(\tau),\ldots,h_p(\tau)$ are different solutions (for a fixed $\tau\in \mathfrak{h}$), we have
\begin{align}\label{eq:prodh}
\prod_{i} (h_i(\tau)-\lambda) = (\tau-\lambda)^p \frac{(f|_k\gamma)(\tau)}{f(\tau)},
\end{align}
where as always $\gamma=\left(\begin{smallmatrix} a & b \\ c & d \end{smallmatrix}\right)\in \sltwoz$ with $\lambda=-\frac{d}{c}$.

We now study the invariance of the solutions $h$ under $\sltwoz$.
\begin{prop}\label{prop:hequiv}
If $h:\mathcal{F}\to \c$ is a solution of~\eqref{eq:defh} for $\tau=\tau_0$, then
\begin{enumerate}[{\upshape (i)}]
\item  $\displaystyle \gamma h=\frac{a h+b}{c h+d}:\mathcal{F}\to \c$ is a solution of~\eqref{eq:defh} for $\tau=\gamma\tau_0$. 
\item $-\overline{h}: \mathcal{F}\to \c$ is a solution of~\eqref{eq:defh} for $\tau=-\overline{\tau_0}$.
\end{enumerate}
\end{prop}
\begin{proof}
It suffices to show the first part for the two generators $\left(\begin{smallmatrix} 1 & 1 \\ 0  & 1 \end{smallmatrix}\right)$ and $\left(\begin{smallmatrix} 0 & -1 \\ 1  & 0 \end{smallmatrix}\right)$ of $\sltwoz$. For the first generator, the result is almost immediate, so we only have to consider the inversion in the unit disk. This follows from the following computation:
\begin{align}\label{eq:hsl2z} 
\sum_{j=0}^p &
\frac{(\mathfrak{d}^jf)(-\tau^{-1})}{j!} \Bigl(\frac{1}{2\pi\ii}\frac{1}{-\tau^{-1}+h(\tau)^{-1}}\Bigr)^j \\
&= 
\tau^{k}\sum_{j=0}^p\sum_{m=0}^{p-j}  \frac{(\mathfrak{d}^{j+m}f)(\tau)}{j!m!} \Bigl(\frac{1}{2\pi \ii}\frac{1}{-\tau+\tau^2h(\tau)^{-1}}\Bigr)^j \Bigl(\frac{1}{2\pi \ii}\frac{1}{\tau}\Bigr)^m \\
&= 
\tau^{k}\sum_{\ell=0}^p \frac{(\mathfrak{d}^{\ell}f)(\tau)}{\ell!} \Bigl(\frac{1}{2\pi \ii}\frac{-\tau+\tau^2h(\tau)^{-1}+\tau}{\tau(-\tau+\tau^2h(\tau)^{-1})}\Bigr)^\ell  \\
&= 
\tau^{k}\sum_{\ell=0}^p \frac{(\mathfrak{d}^{\ell}f)(\tau)}{\ell!} \Bigl(\frac{1}{2\pi \ii}\frac{1}{-h(\tau)+\tau}\Bigr)^\ell =0.
\end{align}
The second statement follows from the fact that for quasimodular forms $g$ with real Fourier coefficients one has $g(-\overline{\tau})=\overline{g(\tau)}$. \end{proof}

Extend the action of~$\sltwoz$ on~$\mathfrak{h}$ to an action of~$\mathrm{GL}_2(\z)$ by
\[\gamma\tau = \frac{a\overline{\tau}+b}{c\overline{\tau}+d} \qquad \text{if } \det\gamma=-1\]
for $\gamma=\left(\begin{smallmatrix} a & b \\ c & d \end{smallmatrix}\right)\in \mathrm{GL}_2(\z)$ and $\tau \in \mathfrak{h}$. Then, we find that the vector $\vec{h}=(h_1,\ldots,h_p)$ is a meromorphic vector-valued equivariant form for~$\mathrm{GL}_2(\z)$: 
\begin{cor}\label{cor:meromorphic}
Let $U$ be a simply connected open subset of $\mathfrak{h}$ for which the $p$ solutions of~\eqref{eq:defh} are distinct. Then, one can choose solutions $h_1,\ldots,h_p:U\to \c$ of~\eqref{eq:defh} such that
\begin{enumerate}[{\upshape (i)}]
    \item for all $\tau \in U$ and $\gamma\in \mathrm{GL}_2(\z)$ the solutions $h_j(\gamma \tau):\gamma U\to \c$ are meromorphic;
    \item if $\Gamma\leq \mathrm{GL}_2(\z)$ such that $\Gamma U \subseteq U$, then there exists a homomorphism $\sigma:\Gamma\to \mathfrak{S}_p$ such that for all $\tau \in U$ and $\gamma\in \Gamma$ one has\[ h_j(\gamma\tau)=\gamma\, h_{\sigma(\gamma)j}(\tau);\]
    \item\label{cor:meromorphic(iii)} $f$ has a zero at $\gamma\tau$ if and only if $h_j(\tau)=\lambda(\gamma)$ for some $j$. 
\end{enumerate}
\end{cor}
\begin{proof}
By the implicit function theorem, there exist $p$ meromorphic solutions~$h_j$ on~$U$, which, by construction, satisfy the third property. By the previous proposition for all $\gamma \in \mathrm{GL}_2(\z)$ we have $h_j(\gamma\tau)=\gamma\, h_{\sigma(\gamma) j}(\tau)$ for some $\sigma(\gamma) \in \mathfrak{S}_p$, possibly depending on $\tau$. However, by continuity of $h_j$ on $U$, we find $\sigma(\gamma)$ does not depend on $\tau$ for $\gamma \in \Gamma$.  In particular, $h_j:\gamma U\to \c$ is meromorphic and $\sigma$ is easily seen to be a homomorphism. 
\end{proof}

We often make use of the fact that $h_j$ is a vector-valued equivariant function in the following way. 
Write 
\[C=\left(\begin{smallmatrix} -1 & 0 \\ 0 & 1 \end{smallmatrix}\right),\quad
S=\left(\begin{smallmatrix} 0 & -1 \\ 1 & 0 \end{smallmatrix}\right),\quad
T=\left(\begin{smallmatrix} 1 & 1 \\ 0 & 1 \end{smallmatrix}\right)
\]
for complex \emph{C}onjugation, \emph{S}piegeln (reflecting) in the unit disc and \emph{T}ranslation. Then,
\[
\mathrm{PGL}_2(\z) = \langle C,S,T \mid C^2=1,(CT)^2=1,(CS)^2=1, S^2=1, (ST)^3=1\rangle.
\]
Given $\gamma \in \mathrm{PGL}_2(\z)$, let $\c^\gamma$ be the set of $\tau\in \c$ such that $\gamma \tau=\tau$. Then,
for all $\tau \in U\cap \c^\gamma$ we have
\[ h_j(\tau) = h_j(\gamma\tau) = \gamma h_{\sigma(\gamma)j}(\tau).\]
Hence, if $\sigma(\gamma)=e$, then $h_{j}(\tau)\in \c^\gamma.$ For example, as a corollary we get the following.
\begin{lem}\label{lem:equivariance} 
Given the solutions $h_j$ and $\sigma:\Gamma\to \mathfrak{S}_p$ as in \cref{cor:meromorphic}, write $\Gamma_j=\{\gamma \in \Gamma \mid \sigma(\gamma)j=j\}$. Then,
\begin{enumerate}[{\upshape (i)}]
    \item\label{it:R} $h_j(\frac{n}{2}+\r\mathrm{i}) \in \frac{n}{2}+\r\mathrm{i}$ if $n\in \z$ is such that $CT^{-n}\in \Gamma_j$\,,
    \item\label{it:C} $|h_j(\tau)|=1$ for $|\tau|=1$ if $CS\in \Gamma_j$\,,
    \item\label{it:rho} $h_j(\rho)\in \{\pm \rho\}$ if $ST\in \Gamma_j$\,.
\end{enumerate}
\end{lem}

\paragraph{Depth 1}
For a quasimodular form $f=f_0+E_2f_1$ of depth~$1$, we have $h:\mathfrak{h}\to \c$ is a holomorphic equivariant function, i.e.,
\[ h(\gamma\tau) = \gamma h(\tau)\]
for all $\tau\in \mathfrak{h}$ and $\gamma \in \mathrm{GL}_2(\z)$. In fact, we can write~$h$ as 
\[ h(\tau) \=  \tau + \frac{12}{2\pi \ii}\frac{f_1(\tau)}{f(\tau)} \= \tau\frac{f|S(\tau)}{f(\tau)} .\]

\begin{remark} There are several additional interesting properties of equivariant functions~$h$ of which we do not make use in this work. Among those are: \vspace{-5pt}
\begin{enumerate}[{\upshape(i)}]\itemsep0pt
    \item The Schwarzian~derivative~$\displaystyle \{h(\tau),\tau\}$ is a meromorphic modular form of weight~$4$. This follows directly from the properties of the Schwarzian~derivative, see, e.g., \cite{ES12}.
    \item The derivative $h'$ equals $Ff^{-2}$ for some (holomorphic) modular form $F$ of weight~$2k$ (where $k$ is the weight of $f$). Explicitly,
    \[ F \= f_0^2 \+ 12\,f_0\,\vartheta(f_1)\meno 12\,\vartheta(f_0)\,f_1 +f_1^2\,E_4\,,\]
    where $\vartheta$ denotes the Serre~derivative. This was observed and proven in \cite[Section~5.3]{GO20} in the case~$f$ is the derivative of a modular form. \qedhere
\end{enumerate}
\end{remark}

\paragraph{Example in depth 1}
\begin{exmp}
If $f=g'=\frac{1}{2\pi \ii}\pdv{g}{\tau}$ with $g$ a modular form of weight~$k$, then
$h(\tau) = \tau + \frac{k}{2\pi\ii}\frac{g(\tau)}{g'(\tau)}$ (to the study of which \cite{ES12} is devoted). In particular, for $f=\Delta'=\Delta E_2$ we find
$h(\tau) = \tau+ \frac{12}{2\pi \ii} \frac{1}{E_2},$ which is also the equivariant function associated to $E_2$. 
\end{exmp}

\paragraph{Examples in higher depth}

\begin{exmp}\label{ex:h2}
Consider $f = E_4 + E_2^2$. Then the $h_j$ for $j=1, 2$ are solutions of the equation
$$0 = (2\pi \ii)^2(E_4(\tau) + E_2(\tau)^2)(\tau - h(\tau))^2 +  48  \pi \ii E_2(\tau) (\tau - h(\tau)) + 144. $$
From this
$$h_j(\tau) = \tau + \frac{12}{2 \pi \ii} \frac{E_2(\tau) +(-1)^j\sqrt{ -E_4(\tau)}}{(E_4(\tau) + E_2(\tau)^2)},$$
which (for an appropriate choice of the square root) is meromorphic on every simply connected domain~$U$ not containing an $\sltwoz$-translate of~$\rho$. For example, for
\[ U\=\mathrm{int}(\mathcal{F}\cup S\mathcal{F}) \= \{ z\in \mathfrak{h} \mid -\tfrac{1}{2}<\Re(z)<\tfrac{1}{2}, |z-1|>1 \text{ and } |z+1|>1\} \]
we have that $\Gamma=\langle C,S\rangle$ satisfies $\Gamma U=U$ and $\sigma:\Gamma\to \mathfrak{S}_2$ is given by
$\sigma(C) = (1\,2), \sigma(S)=(1\,2).$ In particular, $CS\in \ker \sigma$. 
Hence, by \cref{lem:equivariance} we have $|h_j(z)|=1$ if $|z|=1$. However, it is not the case that, e.g., $h_j(-\frac{1}{2}+\r\ii)\in -\frac{1}{2}+\r\ii$.  
\end{exmp}

\section{Zeros in other fundamental domains \texorpdfstring{$(\lambda<\infty)$}{}}\label{sec:5}
By definition of the functions $h_i$ we have
\begin{equation}
N_\lambda(f)-N_\infty(f) \= \frac{1}{2 \pi \ii}\sum_{i}\int_{\partial\mathcal{F}} \frac{h_i'(\tau)}{h_i(\tau)-\lambda} \,\dd \tau \= \frac{1}{2 \pi}\sum_{i}\mathrm{Im}\int_{\partial\mathcal{F}} \frac{h_i'(\tau)}{h_i(\tau)-\lambda} \,\dd \tau ,
\end{equation}
where the second equality holds as the number of zeros of a function is a real number. Hence, the value of~${N_\lambda(f)-N_\infty(f)}$ is determined by the variation of the argument of the $h_i-\lambda$. 
\begin{proof}[Proof of \cref{thm:main1}] Given $\lambda\in \r$, we consider the variation of the argument of $h_j(\tau)-\lambda$ for all $j$. Note that, when moving along $\tau\in\partial\mathcal{F}$, by \cref{cor:meromorphic} the functions $h_i$ are continuous and piecewise meromorphic. 
We change the contour $\partial \mathcal{F}$ to a family of contours $\mathscr{C}_\epsilon$ such that for $\epsilon>0$ sufficiently small there are no poles on the contour of integration, and that the functions $h_j$ only take finitely many real values. (For example, we could define $\mathscr{C}_\epsilon$ to be the shift of the contour $\partial \mathcal{F}$ by $(|\Re(z)|+\sqrt{3}\Re(z)\ii)\epsilon+\frac{1}{2}\epsilon^2$. Note that for $\epsilon\to 0$ the value of the integral over the shifted contour converges to the desired value~$N_\lambda$.)

The functions~$\tau\mapsto h_j(\tau)$ intersect the real axis only a finite number of times as $\tau$ goes over the (shifted) contour.
Write $\lambda_1(\epsilon)<\ldots<\lambda_{n(\epsilon)}(\epsilon)$ for the intersection points for all functions $h_1,\ldots,h_p$ (here $n(\epsilon)$ may also depend on $\epsilon$). Moreover, write $\lambda_1,\ldots,\lambda_n$ for the limiting values of $\lambda_i(\epsilon)$ as $\epsilon\to 0$.  As $h_j-\lambda$ is just a horizontal shift of $h_j$,  given $\lambda,\lambda'\in \r$, the functions $h_j-\lambda$ and $h_j-\lambda'$ admit the same variation of the argument if there is no $\ell$ such that $\lambda<\lambda_\ell<\lambda'$ or $\lambda'<\lambda_\ell<\lambda$. 
Hence, in that case $N_\lambda-N_\infty=N_{\lambda'}-N_\infty$. 
Moreover, for $\lambda>\lambda_n$ the variation of the argument is~$0$. Hence,
$N_\lambda - N_\infty=0$ for $\lambda>\lambda_n$. We conclude that if we define the elements of $\mathscr{I}$ to be 
\[ (-\infty,\lambda_1),\{\lambda_1\},(\lambda_1,\lambda_2),\ldots,\{\lambda_n\},(\lambda_n,\infty)\]
the statement follows. (In case $\lambda_i=\pm\infty$ simply leave out the corresponding sets.)
\end{proof}

\paragraph{Depth 1}
For quasimodular forms of depth $1$, by \cref{lem:equivariance} we have
\begin{itemize}
    \item $h(\tfrac{1}{2}+it) \in \tfrac{1}{2}+i\r$ for $t\in \r$;
    \item $|h(\tau)|=1$ if $|\tau|=1$.
\end{itemize}
Hence, the only possible values of $\lambda_i$ in the above proof are $\pm \frac{1}{2}, \pm 1$ and $\pm \infty$. Therefore, we obtain the following corollary of \cref{thm:main1}.
\begin{cor}
\label{cor:5.2}
For a quasimodular form of depth $1$ with real Fourier coefficients, there exist constants $N_{[0,\frac{1}{2})}(f), N_{(\frac{1}{2},1)}(f), N_{(1,\infty)}(f)$ such that
\[ N_\lambda(f) = 
\begin{cases}
N_{[0,\frac{1}{2})}(f) & |\lambda|\in (0,\frac{1}{2})\\
N_{(\frac{1}{2},1)}(f) & |\lambda|\in(\frac{1}{2},1)\\
N_{(1,\infty)}(f) & |\lambda|\in (1,\infty).
\end{cases}
\]
\end{cor}

Next, by relating $N_{\lambda}(f)$ to $N_{-\frac{1}{\lambda}}(f)$ and by using the above properties of $h$, we prove the following statement, which finishes the proof of \cref{thm:main2}. Recall $z_1,\ldots, z_m$ are the zeros of~$f_1$ such that $\Re z_i=-\frac{1}{2}$ and $\Im z_i>\frac{1}{2}\sqrt{3}$, counted with multiplicity and ordered by imaginary part, and $z_0=\rho$. 
Moreover, recall $r(f_1)$ denotes the sign of the first non-zero Taylor coefficient of~$f_1$ (see~\eqref{eq:r}), and  
$s(f)=\sgn a_0(f)$ if~$f$ does not vanish at infinity, and $s(f)=-\sgn a_0(f_1)$ else. 
Also, $w(z_0)=2$ if $z_0$ equals $\rho,\ii$ or $-\frac{1}{2}+\ii \infty$, and $w(z)=1$ for all other $z\in \mathcal{F}$.
\begin{thm} 
\label{thm:5.3}
Let $f=f_0+E_2f_1$ be an irreducible quasimodular form of depth $1$. If $\lambda\in \r$ with $\frac{1}{2}<|\lambda|<2$, then
\[N_\lambda(f)+N_{-\frac{1}{\lambda}}(f) \= \biggl \lfloor \frac{k}{6} \biggr \rfloor.\]
Moreover, if $|\lambda|<\frac{1}{2}$ or $|\lambda|>2$ we have
\begin{equation}
\label{EqnThm:5.3}
N_\lambda(f)+N_{-\frac{1}{\lambda}}(f)
\=\biggl\lceil\frac{k}{6}\biggr\rceil \meno  r(f_1)\sum_{j=0}^{m} \frac{(-1)^j}{w(z_j)}\, \sgn f(z_j) - \frac{1}{2}(-1)^{m+1}r(f_1)\,s(f).
\end{equation}
\end{thm}

\noindent\textit{Proof.}
As before, we have 
\begin{equation}\label{eq:integral}
N_\lambda(f)+N_{-\frac{1}{\lambda}}(f)-2N_\infty(f) = \frac{1}{2 \pi}\Im \int_{\partial\mathcal{F}} \frac{h'(\tau)}{h(\tau)-\lambda}+\frac{h'(\tau)}{h(\tau)+\frac{1}{\lambda}}  \,\dd \tau .
\end{equation}

\begin{wrapfigure}{r}{0.37\textwidth} 
    \centering \vspace{-10pt}
    \includegraphics[clip, trim=2.5cm 19.1cm 13.8cm 3.5cm]{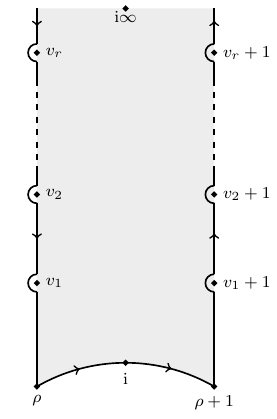}\vspace{-7pt}
    \caption{Contour of integration}\label{fig:contour}\vspace{-15pt}
\end{wrapfigure}

We split this integral in several pieces, and compute the contribution of each piece separately. This argument resembles the one of \cref{prop:Ninf}, as well as the proof found in \cite[Section~5.6]{GO20}. 

\step{Setup}

Let $z_1, \ldots, z_m$ be the zeros of $f_1$ on $\mathcal{L} \backslash \{ \rho \}$ and let $v_1, \ldots, v_r$ be the zeros of $f$ on $\mathcal{L}$, all ordered by imaginary part. We assume $r\geq 1$, and at the end of the argument verify the proof goes through if $r=0$. Moreover, without loss of generality, we assume $|\lambda|>1$. 

The finite poles of~$h$ are exactly the finite zeros of~$f$.
We fix $\epsilon>0$ sufficiently small. Let $L_\epsilon$ be the punctured left half line
\[
[\rho, v_1 - \ii \epsilon]\, \cup\,[v_1 + \ii \epsilon, v_2 - \ii \epsilon] \,\cup \,\ldots \,\cup\, [v_{r}+\ii \epsilon, -\tfrac{1}{2} + \ii \infty],
\]
and define the punctured right half line by $R_\epsilon=L_\epsilon+1$. The line segment $[\rho, v_1 - \ii \epsilon]$ (as well as its shift on $R_\epsilon$) is referred to as the \emph{lower vertical segment}, whereas  $[v_{r}+\ii \epsilon, -\frac{1}{2} + \ii \infty]$ a called an \emph{upper vertical segment}.

Recall that by definition of $N_\lambda$, we include all zeros/poles of $h(\tau)-\lambda$ and $h(\tau)+\frac{1}{\lambda}$ on $\mathcal{L}$ in the integral~\eqref{eq:integral}, but not those on $\mathcal{R}$. Hence, for each zero $v_i$ we introduce a semicircle $C_i$ around $v_i$ of radius $\epsilon$ on the left of $\mathcal{L}$, as well as the semicircle $C_i+1$. The boundary of $\mathcal{F}$ then consists of $\mathcal{C}, L_\epsilon, R_\epsilon$ and the semicircles $C_i$ and $C_i+1$ for all $i$. 

\step{Circular segment~$\mathcal{C}$}
Recall $|h(\tau)|=1$ if $|\tau|=1$. Hence, $h$ has no poles on $\mathcal{C}$. 
Using the expression~\eqref{eq:prodh} above, we write
\begin{align}
\frac{1}{2 \pi}&\Im\int_{\mathcal{C}} \frac{h'(\tau)}{h(\tau)-\lambda}+\frac{h'(\tau)}{h(\tau)+\frac{1}{\lambda}}  \,\dd \tau \\
&= \frac{1}{2 \pi }\mathrm{Im}\int_{\mathcal{C}} \pdv{}{\tau}\log \bigl((h(\tau)-\lambda)(h(\tau)+\tfrac{1}{\lambda})\bigr) \,\dd \tau  \\
&= \frac{1}{2 \pi}\mathrm{Im}\int_{\mathcal{C}} \pdv{}{\tau} \Bigl((\tau-\lambda)^{1-k}(\tau+\frac{1}{\lambda})^{1-k} \frac{f(\gamma\tau)f(\gamma S\tau)}{f(\tau)^2}\Bigr) \,\dd \tau,
\end{align}
where $\gamma=\left(\begin{smallmatrix} a & b \\ c & d \end{smallmatrix}\right)\in \sltwoz$ is such that $\lambda=-\frac{d}{c}$, and $S=\left(\begin{smallmatrix} 0 & 1 \\ -1 & 0 \end{smallmatrix}\right)$. Note that $\gamma S = \left(\begin{smallmatrix} -b & a \\ -d & c \end{smallmatrix}\right)$ and $-\frac{1}{\lambda} = -\frac{c}{-d}.$

First, we compute (recall $|\lambda|>1$ by assumption)
\begin{align}
\frac{1}{2 \pi }&\mathrm{Im}\int_{\mathcal{C}} \pdv{}{\tau}\log\Bigl((\tau-\lambda)^{1-k}\bigl(\tau+\frac{1}{\lambda}\bigr)^{1-k} \Bigr) \,\dd \tau \\
&=\frac{k-1}{2 \pi}\mathrm{Re}\int_{\pi/3}^{2\pi/3} -\frac{\frac{1}{\lambda}e^{\ii \theta}}{1-\frac{1}{\lambda}e^{\ii\theta}}+\frac{1}{1+\frac{1}{\lambda}e^{-\ii \theta}} \dd \theta \\
&= \frac{k-1}{2 \pi}\mathrm{Re}\int_{\pi/3}^{2\pi/3} \Bigl( 1 - 2\sum_{n \geq 1} \frac{1}{\lambda^{2n-1}} \cos((2n+1)\theta) + 2\ii \sum_{m \geq 1} \frac{1}{\lambda^{2m}} \sin(2m \theta) \Bigr) \dd \theta\\
& = \frac{k-1}{6}. 
\end{align}

Next, we show that the following expression is actually independent of $\gamma$:
\begin{align}
\frac{1}{2 \pi}&\mathrm{Im}\int_{\mathcal{C}} \pdv{}{\tau}\log\Bigl(\frac{f(\gamma\tau)f(\gamma S\tau)}{f(\tau)^2}\Bigr) \,\dd \tau \\
&= 
\frac{1}{2 \pi}\mathrm{Im}\int_{\mathcal{C}} \frac{1}{(c\tau+d)^2}\frac{f'(\gamma\tau)}{f(\gamma\tau)} + \frac{1}{(-d\tau+c)^2}\frac{f'(\gamma S\tau)}{f(\gamma S\tau)} - 2\frac{f'(\tau)}{f(\tau)} \,\dd \tau.
\end{align}
Applying the coordinate transformation $\tau\mapsto -\frac{1}{\tau}$ to the second term in the integrand, and using~\eqref{eq:Ninfty1} for the last term, this equals
\begin{align}
\frac{1}{2 \pi}\mathrm{Im}\int_{\mathcal{C}} \frac{1}{(c\tau+d)^2}\frac{f'(\gamma\tau)}{f(\gamma\tau)} - \frac{1}{(c\tau+d)^2}\frac{f'(\gamma \tau)}{f(\gamma \tau)} - 2\frac{f'(\tau)}{f(\tau)}\,\dd \tau 
&=-2 N_\infty(f).
\end{align}

Hence, the contribution of $\mathcal{C}$ equals
\[ \frac{k-1}{6} -2 N_\infty(f).\]

\clearpage
\step{Lower vertical segments}
Now, first assume that $f_1$ admits no zeros on the lower vertical segment~$[\rho, v_1 - \ii \epsilon]$.
By \cref{lem:equivariance}(\ref{it:R}) we know that $h(\tfrac{1}{2}+\ii t) \in \tfrac{1}{2}+\ii \r$ for all $t \in \r$. Also, by \cref{lem:equivariance}(\ref{it:rho}) we have $h(\rho)\in \{\pm \rho\}$.
Hence, as $\tau$ moves on $[\rho, v_1 - \ii \epsilon]$ the function $h(\tau)$ moves from $\rho$ or $\rho^2$ to $\ii \infty$ or $-\ii\infty$. We have three possibilities for the combined sign of $\Im h(\rho)$ and $\Im \bigl( h(\rho+\epsilon\ii)-\rho \bigr)$, depicted in~\cref{tab:1}. (Note that if the first is negative, the second is necessarily negative as well. As in this segment there are no zeros of $f_1$ or $f$, we know that in the case~(--,--) eventually $h$ tends to $-\ii\infty$.).
\begin{table}\begin{center}
\begin{tabular}{c c c l}
(+,+) & (+,--) & (--,--) \\\hline\\[-5pt]
\begin{tikzpicture}[scale=0.8]
\draw[thick, ->] (-0.5,1.8) -- (-0.5,1.5);
\draw[thick] (-0.5,0.866) -- (-0.5,1.5);
\draw[thick] (0.5,1.8) -- (0.5,1.5);
\draw[thick,->] (0.5,0.866) -- (0.5,1.5);
\draw[dashed] (0.5,0.866) -- (0.4,0);
\draw[dashed] (-0.5,0.866) -- (0.4,0);
\draw[dashed] (0.5,0.866) -- (5/2,0);
\draw[dashed] (-0.5,0.866) -- (5/2,0);
\node[circle,fill=black,inner sep=0pt,minimum size=3pt,label=above left:{$\rho$}] at (-0.5,0.866) {};
\node[circle,fill=black,inner sep=0pt,minimum size=3pt,label=above right:{$\rho+1$}] at (0.5,0.866) {};
\node[circle,fill=black,inner sep=0pt,minimum size=3pt,label=below:{$\frac{1}{\lambda}$}] (a) at (0.4,0) {};
\node[circle,fill=black,inner sep=0pt,minimum size=3pt,label=below:{$\lambda$}] (a) at (5/2,0) {};
\node[circle,fill=black,inner sep=0pt,minimum size=0pt,label=below left:{$\phantom{\rho}$}] at (-0.5,-0.866) {};
\end{tikzpicture}&
\begin{tikzpicture}[scale=0.8]
\draw[thick, ->] (-0.5,-1.8) -- (-0.5,-1.5);
\draw[thick] (-0.5,0.866) -- (-0.5,-1.5);
\draw[thick] (0.5,-1.8) -- (0.5,-1.5);
\draw[thick,->] (0.5,0.866) -- (0.5,-1.5);
\draw[dashed] (0.5,0.866) -- (0.4,0);
\draw[dashed] (-0.5,0.866) -- (0.4,0);
\draw[dashed] (0.5,0.866) -- (5/2,0);
\draw[dashed] (-0.5,0.866) -- (5/2,0);
\node[circle,fill=black,inner sep=0pt,minimum size=3pt,label=above left:{$\rho$}] at (-0.5,0.866) {};
\node[circle,fill=black,inner sep=0pt,minimum size=3pt,label=above right:{$\rho+1$}] at (0.5,0.866) {};
\node[circle,fill=black,inner sep=0pt,minimum size=3pt,label=below left:{$\frac{1}{\lambda}$}] (a) at (0.4,0) {};
\node[circle,fill=black,inner sep=0pt,minimum size=3pt,label=below:{$\lambda$}] (a) at (5/2,0) {};
\node[circle,fill=black,inner sep=0pt,minimum size=0pt,label=below left:{$\phantom{\rho}$}] at (-0.5,-0.866) {};
\end{tikzpicture}&
\begin{tikzpicture}[scale=0.8]
\draw[thick, ->] (-0.5,-1.8) -- (-0.5,-1.5);
\draw[thick] (-0.5,-0.866) -- (-0.5,-1.5);
\draw[thick] (0.5,-1.8) -- (0.5,-1.5);
\draw[thick,->] (0.5,-0.866) -- (0.5,-1.5);
\draw[dashed] (0.5,-0.866) -- (0.4,0);
\draw[dashed] (-0.5,-0.866) -- (0.4,0);
\draw[dashed] (0.5,-0.866) -- (5/2,0);
\draw[dashed] (-0.5,-0.866) -- (5/2,0);
\node[circle,fill=black,inner sep=0pt,minimum size=3pt,label=below left:{$\rho^2$}] at (-0.5,-0.866) {};
\node[circle,fill=black,inner sep=0pt,minimum size=3pt,label=below right:{$\rho^2+1$}] at (0.5,-0.866) {};
\node[circle,fill=black,inner sep=0pt,minimum size=3pt,label=above:{$\frac{1}{\lambda}$}] (a) at (0.4,0) {};
\node[circle,fill=black,inner sep=0pt,minimum size=3pt,label=above:{$\lambda$}] (a) at (5/2,0) {};
\end{tikzpicture}
\\[5pt]
\end{tabular}\end{center}\vspace{-5pt} 
\caption{The three possibilities for the graph of $h(\tau)$ for $\tau$ in a small neighborhood of~$\rho$ (resp.\ $\rho+1$) along the left (resp.\ right) vertical line segment of $\partial \mathcal{F}$, given
$\displaystyle\bigl(\sgn\Im h(\rho),\sgn \lim\nolimits_{\epsilon \downarrow 0} (\Im h(\rho+\epsilon\ii)-\rho)\bigr)\in \{\pm\}^2$.}\vspace{5pt}
\label{tab:1}
\end{table}

Hence, the variation of the argument of $h-\lambda$ along $[\rho, v_1 - \ii \epsilon]$ and $[\rho, v_1 - \ii \epsilon]+1$ equals the (oriented) angle between $h(\rho+1), \lambda$ and $h(\rho)$, as shown in the same table. 
In particular, it is an exercise in Euclidean geometry that the sum of the oriented angles for~$\lambda$ and~$-\frac{1}{\lambda}$ only depends on whether $|\lambda|<2$ or not (recall $|\lambda|>1$ by assumption). In \cref{tab:2} we displayed the contribution of each of the cases in \cref{tab:1}.
\begin{table}\begin{center}
\begin{tabular}{l|cccc}
& (+,+) & (+,--)  &\hspace{1pt} (--,--) \\[2pt]\hline\\[-5pt]
$|\lambda|>2$ & $\frac{1}{6} $ & $-\frac{5}{6}$ 
 & $-\frac{1}{6}$ \\[5pt]\hline\\[-5pt]
$1 <|\lambda|<2$ & $\frac{1}{6} $ & $\phantom{-}\frac{1}{6}$ 
& $-\frac{1}{6}$ 
\end{tabular}
\end{center}
\caption{Variation of the argument of the lower vertical segments in several cases (see \cref{tab:1}).}
\label{tab:2}
\end{table}
Correspondingly, the contribution to the variation of the argument equals 
\[ \frac{1}{6} \sgn \Im h(\rho) + \delta_{|\lambda|>2}\Bigl(-\frac{1}{2} \sgn \Im h(\rho) + \frac{1}{2} \sgn \lim\nolimits_{\epsilon \downarrow 0} (\Im h(\rho+\epsilon\ii)-\rho)\Bigr) \]

Observe that
\[\sgn \lim\nolimits_{\epsilon \downarrow 0} (\Im h(\rho+\epsilon\ii)-\rho) \= - \sgn \lim\nolimits_{\epsilon \downarrow 0} \sgn( f_1(\rho+\epsilon\ii)) \sgn(f(\rho))\]
as $h(\tau) \=  \tau + \frac{12}{2\pi \ii}\frac{f_1(\tau)}{f(\tau)}$. 
Hence, under the assumption $f_1$ has no zeros on the lower vertical segment and by \cref{lem:r}, its contribution equals
\[ \frac{1}{6} \sgn \Im h(\rho)\+\delta_{|\lambda|>2} \Bigl( -\frac{1}{2} \sgn \Im h(\rho) \meno  \frac{r(f_1)}{2} \sgn f(\rho) \Bigr) . \]

Now, suppose $f_1$ admits $p$ zeros on the lower vertical segment. Then, as $f$ admits no zeros on this segment, $h(\tau)$ crosses the line $\tau$ precisely $p$ times. If $p$ is even, this does not alter the variation of the argument, but if $p$ is odd, then $h(\tau)$ changes sign if $\tau$ tends to the pole $v_1$. 

Note that this does not affect the variation of the argument of $h(\tau)-\mu$ if $|\mu|>1$. Suppose $I$ is a line segment of $\mathcal{L}$ for which $h(\tau)-\mu$ tends to $\pm \ii \infty$ or to $0$ on the two boundary points of $I$. In that case the variation of the argument on $I$ equals $v$, whereas the variation of the argument on the corresponding line segment on $\mathcal{R}$ is $-v$. We conclude that if $|\mu|>1$, the only contribution for the variation of the argument is displayed in \cref{tab:2}. 

Hence, the contribution of the lower vertical segment equals
\[\frac{1}{6} \sgn \Im h(\rho) \+\delta_{|\lambda|>2} \Bigl( -\frac{1}{2} \sgn \Im h(\rho) \meno  \frac{r(f_1)}{2} \sgn f(\rho)  \meno r(f_1)\sum_{j} (-1)^j \sgn f(z_j) \Bigr), \]
where the sum is over all $j$ such that $z_j$ lies in the lower vertical segment. Note that by the factor $r(f_1)(-1)^j$ we keep track of the sign of $f_1$ at $z_j$.

\step{Vertical segments between two poles}
Similar as in the previous case (now there are no boundary terms), we find
\[ -\delta_{|\lambda|>2}\,r(f_1) \sum_{j} (-1)^j \sgn f(z_j) \]
where the sum is over all $j$ such that $z_j$ lies between two poles of $f$.

\step{Semicircles centered at the poles}
Note that for sufficiently small $\epsilon$, we have that the value $h-\lambda$ on these semicircles is arbitrarily large (say, bigger than $|\lambda|+1$ in absolute value). Moreover, the contours of $h-\lambda$ on such a semicircle $C_i$ and $C_i+1$ differ only by $1$. Hence, as $\epsilon\to 0$, the contributions of the corresponding semicircles $C_i$ and $C_i+1$ of $\partial\mathcal{F}$ (which admit an opposite orientation) cancel in pairs.

\step{Upper vertical segments}
As before, the variation of the argument vanishes, except for $h(\tau)+\frac{1}{\lambda}$ if $|\lambda|>2$. Again, in this case, we have the contribution
\[ -\delta_{|\lambda|>2}\,r(f_1) \sum_{j} (-1)^j \sgn f(z_j), \]
where the sum is over all $j$ such that $z_j$ lies in the upper vertical segment. This is also the only contribution, except in the following exceptional case. It may be that $h(\tau)$ is smaller than $\tau$ in imaginary value for $\tau=\frac{1}{2}+\ii t$ with $t$ tending to infinity, but still converging to $\ii \infty$. If this is the case, we have to add a contribution of $+1$. By considering the Fourier expansion of $h(\tau)-\tau$, we see this can only happen if $f$ has no zero at infinity (else, $h$ goes to $\pm \ii \infty$ at exponential rate). Moreover, the first non-zero Fourier coefficient of $f_1/f$ should be positive if $h(\tau)-\tau\leq 0$ for $\tau=-\frac{1}{2}+\ii t$. In case $f$ does not vanish at infinity, we have
\[ \lim_{t\to \infty}\sgn\Im(h(-\tfrac{1}{2}+\ii t)-(-\tfrac{1}{2}+\ii t)) =- \lim_{t\to \ii\infty}\sgn(f_1(-\tfrac{1}{2}+\ii t)\, f(-\tfrac{1}{2}+\ii t)),\]
Note that $\lim_{t\to \infty}\sgn(f_1(-\tfrac{1}{2}+\ii t) = (-1)^m r(f_1)$.
We found that this special contribution equals
\[ \begin{cases}
\frac{1}{2}+\frac{1}{2} (-1)^m r(f_1) \lim_{t \to \infty} \sgn f(-\frac{1}{2} + \ii t ) 
& \text{ if } f \left(-\frac{1}{2} + \ii \infty \right) \neq 0\\
0 & \text{ if } f \left(-\frac{1}{2} + \ii \infty \right) = 0, \\
\end{cases}
\]
As $\lim_{t \to \infty} \sgn f(-\frac{1}{2} + \ii t ) = \sgn a_0(f)$ if $f$ does not vanish at the cusp, and $(-1)^{m}r(f_1)=\sgn a_0(f_1)$, by definition of $s(f)$, we find that the total contribution of the upper vertical segment for $|\lambda|>2$ equals
\[-r(f_1)\sum_{j=1}^m (-1)^j \sgn f(z_j) +\frac{1}{2}- \frac{1}{2} (-1)^{m+1} r(f_1)\,s(f)  
\]

\step{Total contribution}
Adding all contributions for $1<|\lambda|<2$, we obtain
$$\frac{k-1}{6}+\frac{1}{6} \sgn \Im h(\rho).$$
Note that $f_1(\rho)=0$ if $k\equiv 0 \, (6)$ and $f_0(\rho)=0$ if $k\equiv 2 \, (6).$ Hence,
\[ h(\rho) = \rho + \begin{cases} 0 & k\equiv 0 \, (6) \\ \frac{12}{2\pi \ii}\frac{1}{E_2(\rho)} & k\equiv 2\, (6).\end{cases}\]
Therefore,
\[ \sgn \Im h(\rho) \= \begin{cases} 1 & k \equiv 0 \, (6) \\ -1 & k \equiv 2 \, (6),\end{cases}
\]
We conclude that for $1<|\lambda|<2$, the variation of the argument equals $\lfloor \frac{k}{6}\rfloor.$

Adding all contributions for $|\lambda|>2$, we obtain
\[
\biggl\lceil\frac{k}{6}\biggr\rceil - \frac{r(f_1)}{2} \sgn f(\rho) -r(f_1)\sum_{j=1}^m (-1)^j \sgn f(z_j) - \frac{1}{2} (-1)^{m+1} r(f_1)\,s(f).
\]

\step{$f$ has no zeros on $\mathcal{L}$}
In this case, $h$ has no poles on $\mathcal{L}$ and $\mathcal{R}$. Therefore, the variation of the argument along $\mathcal{L}$ and $\mathcal{R}$ only depends on the values $h(\rho)$ and $h(-\frac{1}{2} + \ii \infty)$.
The image of $h$ on~$\mathcal{L}$ and~$\mathcal{R}$  is summarized in \cref{tab:3}.
\begin{table}\begin{center}
\begin{tabular}{c c c c}
(+,+) & (+,--) & (--,+) & (--,--) \\\hline\\[-5pt]
\begin{tikzpicture}[scale=0.8]
\draw[thick, ->] (-0.5,1.8) -- (-0.5,1.5);
\draw[thick] (-0.5,0.866) -- (-0.5,1.5);
\draw[thick] (0.5,1.8) -- (0.5,1.5);
\draw[thick,->] (0.5,0.866) -- (0.5,1.5);
\draw[dashed] (0.5,0.866) -- (0.4,0);
\draw[dashed] (-0.5,0.866) -- (0.4,0);
\draw[dashed] (0.5,0.866) -- (5/2,0);
\draw[dashed] (-0.5,0.866) -- (5/2,0);
\node[circle,fill=black,inner sep=0pt,minimum size=3pt,label=above left:{$\rho$}] at (-0.5,0.866) {};
\node[circle,fill=black,inner sep=0pt,minimum size=3pt,label=above right:{$\rho+1$}] at (0.5,0.866) {};
\node[circle,fill=black,inner sep=0pt,minimum size=3pt,label=below:{$\frac{1}{\lambda}$}] (a) at (0.4,0) {};
\node[circle,fill=black,inner sep=0pt,minimum size=3pt,label=below:{$\lambda$}] (a) at (5/2,0) {};
\node[circle,fill=black,inner sep=0pt,minimum size=0pt,label=below left:{$\phantom{\rho}$}] at (-0.5,-0.866) {};
\end{tikzpicture}&
\begin{tikzpicture}[scale=0.8]
\draw[thick, ->] (-0.5,-1.8) -- (-0.5,-1.5);
\draw[thick] (-0.5,0.866) -- (-0.5,-1.5);
\draw[thick] (0.5,-1.8) -- (0.5,-1.5);
\draw[thick,->] (0.5,0.866) -- (0.5,-1.5);
\draw[dashed] (0.5,0.866) -- (0.4,0);
\draw[dashed] (-0.5,0.866) -- (0.4,0);
\draw[dashed] (0.5,0.866) -- (5/2,0);
\draw[dashed] (-0.5,0.866) -- (5/2,0);
\node[circle,fill=black,inner sep=0pt,minimum size=3pt,label=above left:{$\rho$}] at (-0.5,0.866) {};
\node[circle,fill=black,inner sep=0pt,minimum size=3pt,label=above right:{$\rho+1$}] at (0.5,0.866) {};
\node[circle,fill=black,inner sep=0pt,minimum size=3pt,label=below left:{$\frac{1}{\lambda}$}] (a) at (0.4,0) {};
\node[circle,fill=black,inner sep=0pt,minimum size=3pt,label=below:{$\lambda$}] (a) at (5/2,0) {};
\node[circle,fill=black,inner sep=0pt,minimum size=0pt,label=below left:{$\phantom{\rho}$}] at (-0.5,-0.866) {};
\end{tikzpicture}&
 \begin{tikzpicture}[scale=0.8]
 \draw[thick, ->] (-0.5,1.8) -- (-0.5,1.5);
 \draw[thick] (-0.5,-0.866) -- (-0.5,1.5);
 \draw[thick] (0.5,1.8) -- (0.5,1.5);
 \draw[thick,->] (0.5,-0.866) -- (0.5,1.5);
 \draw[dashed] (0.5,-0.866) -- (0.4,0);
 \draw[dashed] (-0.5,-0.866) -- (0.4,0);
 \draw[dashed] (0.5,-0.866) -- (5/2,0);
 \draw[dashed] (-0.5,-0.866) -- (5/2,0);
 \node[circle,fill=black,inner sep=0pt,minimum size=3pt,label=below left:{$\rho^2$}] at (-0.5,-0.866) {};
 \node[circle,fill=black,inner sep=0pt,minimum size=3pt,label=below right:{$\rho^2+1$}] at (0.5,-0.866) {};
 \node[circle,fill=black,inner sep=0pt,minimum size=3pt,label=above left:{$\frac{1}{\lambda}$}] (a) at (0.4,0) {};
 \node[circle,fill=black,inner sep=0pt,minimum size=3pt,label=above:{$\lambda$}] (a) at (5/2,0) {};
 \end{tikzpicture}&
\begin{tikzpicture}[scale=0.8]
\draw[thick, ->] (-0.5,-1.8) -- (-0.5,-1.5);
\draw[thick] (-0.5,-0.866) -- (-0.5,-1.5);
\draw[thick] (0.5,-1.8) -- (0.5,-1.5);
\draw[thick,->] (0.5,-0.866) -- (0.5,-1.5);
\draw[dashed] (0.5,-0.866) -- (0.4,0);
\draw[dashed] (-0.5,-0.866) -- (0.4,0);
\draw[dashed] (0.5,-0.866) -- (5/2,0);
\draw[dashed] (-0.5,-0.866) -- (5/2,0);
\node[circle,fill=black,inner sep=0pt,minimum size=3pt,label=below left:{$\rho^2$}] at (-0.5,-0.866) {};
\node[circle,fill=black,inner sep=0pt,minimum size=3pt,label=below right:{$\rho^2+1$}] at (0.5,-0.866) {};
\node[circle,fill=black,inner sep=0pt,minimum size=3pt,label=above:{$\frac{1}{\lambda}$}] (a) at (0.4,0) {};
\node[circle,fill=black,inner sep=0pt,minimum size=3pt,label=above:{$\lambda$}] (a) at (5/2,0) {};
\end{tikzpicture}
\\[5pt]
\end{tabular}\end{center}\vspace{-15pt} 
\caption{The four possibilities for the graph of $h(\tau)$ for $\tau$ along the left (resp.\ right) vertical line segment of $\partial \mathcal{F}$, given
$\displaystyle\bigl(\sgn\Im h(\rho), \lim_{t \to \infty} \Im \sgn h(-\tfrac{1}{2} + \ii t )  \bigr)\in \{\pm\}^2$.}\vspace{15pt}
\label{tab:3}
\end{table}
\begin{table}\begin{center}
\begin{tabular}{l|cccc}
& (+,+) & (+,--)  &\hspace{1pt} (--,+) & (--,--) \\[2pt]\hline\\[-5pt]
$|\lambda|>2$ & $\frac{1}{6} $ & $-\frac{5}{6}$ 
 & $\phantom{-}\frac{5}{6}$ 
 & $-\frac{1}{6}$ \\[5pt]\hline\\[-5pt]
$1 <|\lambda|<2$ & $\frac{1}{6} $ & $\phantom{-}\frac{1}{6}$ 
& $-\frac{1}{6}$ 
& $-\frac{1}{6}$ 
\end{tabular}
\end{center}\vspace{-10pt} 
\caption{Variation of the argument along the left (resp. right) vertical line segment of $\partial \mathcal{F}$ in several cases (see \cref{tab:3}).}
\label{tab:4}
\end{table}
The contributions to the variation of the argument are given in the following \cref{tab:4}. Note that the contributions in this table yield the same final formula for $N_{\lambda}(f) + N_{-\frac{1}{\lambda}}(f)$.  \qed

\begin{proof}[Proof of \cref{thm:crit}] The result follows from \cref{thm:main2}, as we explain now. 
First, let $\phi$ be the unique modular form such that $f:=g'/\phi$ is irreducible (if $g$ has only simple zeros, and no zeros at the cusp, then $\phi=1$). 

Recall that the sign of the derivative of a real-valued differentiable function in two consecutive zeros of this function is opposite. Hence, for two consecutive zeros of~$g$, the function
\[ 
f \= \frac{g'}{\phi} 
\]
changes sign, i.e., $(-1)^j \sgn f(z_j)$ and $(-1)^j \sgn \widehat{f}(\theta_j)$ are independent of $j$. Comparing with the behaviour of $f_1$ around $\rho$ one obtains
\[ (-1)^j \sgn f(z_j) = - r(f_1) = (-1)^j \sgn \widehat{f}(\theta_j)\]
for all $j>0$. 

Moreover, we have
\[\sgn f(\rho) \= \begin{cases}
r(f_1) & k(f)\equiv 2 \mod 6\\
-r(f_1) & k(f)\equiv 0 \mod 6,
\end{cases}  \]
where $k(f)$ denotes the weight of $f$. 
Namely, in case $k(f)\equiv 2 \mod 6$ we have $f_0(\rho)=0$ and $f(\rho) = \sgn f_1(\rho)$, which is non-zero by definition of $\phi$. Moreover, in case $k(f)\equiv 0 \mod 6$, we have 
\[ \sgn f(\rho) \= \sgn\Bigl(\frac{g'}{\phi}(\rho)\Bigr) \= \sgn\Bigl(\frac{(f_1 \phi)'}{\phi}\Bigr) \= \sgn\Bigl(f_1'(\rho)+f_1(\rho)\frac{\phi'}{\phi}(\rho)\Bigr).\]
Note that $\phi(\rho)\neq 0$ if $k(f)\equiv 0 \mod 6$, and $f_1(\rho)=0$. Hence, we find
\[ \sgn f(\rho) \= \sgn f_1'(\rho) \= -r(f_1).\]
Note that $a_0(f)=0$ if and only if $a_0(g)\neq 0$. 
In case $g$ has $n\geq 1$ roots of the cusps, observe that $f$ and $n \frac{g}{\phi} E_2$ have the same (non-zero) constant term. Also, observe that $f_1$ equals $\frac{g}{\phi}$ up to a constant (being the weight of $g$, divided by $12$). Hence, in that case, we have 
\[ s(f) = \sgn(a_0(f)) = \sgn(a_0(g/\phi)) = \sgn(a_0(f_1)) = (-1)^{m} r(f_1),\]
where the last equality holds as $f_1$ does not vanish at infinity. Hence, 
\[-\frac{1}{2}(-1)^{m+1} r(f_1)s(f) \= \begin{cases} -\frac{1}{2} & a_0(g)\neq 0 \\
\frac{1}{2} & a_0(g)=0 .\end{cases}\]

Finally, write $\delta'$ and $\epsilon'$ for the order of $E_4$ and $E_6$ in $g'$. Note that $k(f)\equiv k+2 \mod 6$ if $g$ does not have repeated zeros at $\rho$ and $\ii$ (here $k$ is the weight of $g$, or $k+2$ is the weight of $k'$). More generally, 
\[ k+2 \:\equiv\: k(f)+4\delta'+6\epsilon' \mod 6.\]
Observe the following identity
\[ \frac{1}{2}\biggl \lfloor \frac{k+2}{6} -\frac{2}{3}\delta'\biggr \rfloor \+ \frac{1}{3}\delta' = \frac{k}{12} + \frac{1}{6} \delta_{g(\rho)=0}\,, \]
where $\delta_{g(\rho)=0}$ is $1$ precisely if $g(\rho)=0$ (if not, then $k\equiv 0 \mod 6$).  

We conclude that
\begin{align}
N_{(1,\infty]}(g')  &\= \frac{1}{2}\biggl \lfloor \frac{k+2-4\delta'-6\epsilon'}{6} \biggr \rfloor \+ \frac{2\delta'+3\epsilon'}{6} \+  C(g)\+ \frac{1}{6} \delta_{g(\rho)=0} \\
&\= \frac{k}{12} \+  C(g)\+ \frac{1}{3} \delta_{g(\rho)=0}\,,
\end{align}
\begin{align}
N_{(1,\infty]}(g') + N_{(\frac{1}{2},1)}(g') &\=  \biggl\lfloor \frac{k+2-4\delta'-6\epsilon'}{6} \biggr\rfloor \+ \frac{2\delta'+3\epsilon'}{3} 
\=  \frac{k}{6}+\frac{1}{3}\delta_{g(\rho)=0}
\end{align}
and $N_{(1,\infty]}(g') + N_{[0,\frac{1}{2})}(g')$ equals
\begin{align}
&\biggl\lceil \frac{k+2-4\delta'-6\epsilon'}{6} \biggr\rceil \+ \frac{2\delta'+3\epsilon'}{3}
\+ \frac{1}{2}\delta_{k(f)\equiv 4\,(6)}- \frac{1}{2}\delta_{k(f)\equiv 0\,(6)}+
|\mathcal{L}(g)| - \frac{1}{2} \\
\= &
\biggl\lfloor \frac{k+2}{6}-\frac{2}{3}\delta' \biggr\rfloor \+ \frac{2}{3}\delta'
\+L(g) \\
\= &
\frac{k}{12}
\+L(g)+\delta_{g(\rho)=0}\,. \qedhere
\end{align} 
\end{proof}

\begin{proof}[Proof of \cref{cor:upperbound}]
First of all, for $f$ as in \cref{thm:main2} we have
\[N_{(1,\infty]}(f) \:\leq\: \frac{1}{2}\biggl \lfloor \frac{k}{6} \biggr \rfloor + n' + \frac{1}{2}\delta_{k\equiv 0\,(6)}, \]
where $n'$ counts the weighted number of zeros of $f_1$ on the part of the unit circle with angle $\theta$ such that  $\frac{\pi}{2}\leq \theta<\frac{2\pi}{3}$. Now $n'\leq \frac{k-2}{12}-\frac{1}{3}\delta_{k\equiv 0\,(6)}$, as $f_1$ has precisely $\frac{k-2}{12}$ zeros (of which at least one in $\rho$ if $k\equiv 0 \mod 6$). Hence,
\[N_{(1,\infty]}(f) \:\leq\: \frac{1}{2}\biggl \lfloor \frac{k}{6} \biggr \rfloor + \frac{k-2}{12} + \frac{1}{6}\delta_{k\equiv 0\,(6)} \= \biggl \lfloor \frac{k}{6} \biggr \rfloor \= \dim \widetilde{M}_k^{\leq 1}-1,\]
where we used that $k\equiv 0, 2 \mod 6.$ Now, as $f$ is not divisible by $E_4$ by assumption, this upper bound also holds true if $f$ is reducible. Similarly, we obtain the upper bound for $N_{(\frac{1}{2},1)}(f)$. 

Next, for $f$ as in \cref{thm:main2} we have
\[ N_{(1,\infty]}(f) + N_{[0,\frac{1}{2})}(f) \:\leq \: \biggl\lceil\frac{k}{6}\biggr\rceil + m +\delta_{k\equiv 0\,(6)}. \]
Here, the term $\delta_{k\equiv 0\,(6)}$ comes from the fact that $-r(f_1)\frac{(-1)^0}{2}\sgn f(\rho)$ equals $-\frac{1}{2}$ if $k\equiv 2\,(6)$, as in that case $r(f_1)=\sgn f_1(\rho) = \sgn f(\rho)$. For $k\equiv 0 \,(6)$ we have $-r(f_1)\frac{(-1)^0}{2}\sgn f(\rho)\leq \frac{1}{2}.$
Now, similar as before, $N_{(1,\infty]}(f) \:\geq\: \frac{1}{2} \lfloor \frac{k}{6} \rfloor - n' - \frac{1}{2}\delta_{k\equiv 0\,(6)}.$
Hence, 
\[N_{[0,\frac{1}{2})}(f) \:\leq \: \biggl\lceil\frac{k}{6}\biggr\rceil- \frac{1}{2}\biggl \lfloor \frac{k}{6} \biggr \rfloor + n' + m  +  \frac{3}{2}\delta_{k\equiv 0\,(6)}
\]
Now, similarly, $n'+m\leq \frac{k-2}{12}-\frac{1}{3}\delta_{k\equiv 0\,(6)}$, so that
\[N_{[0,\frac{1}{2})}(f) \:\leq \: \biggl\lceil\frac{k}{6}\biggr\rceil- \frac{1}{2}\biggl \lfloor \frac{k}{6} \biggr \rfloor + \frac{k-2}{12} + \frac{7}{6}\delta_{k\equiv 0\,(6)} \= \biggl \lfloor \frac{k}{6} \biggr \rfloor +1 \= \dim \widetilde{M}_k^{\leq 1}\]
as $k\equiv 0,2\mod 6$. This implies the corollary. 
\end{proof}

\paragraph{Examples}
\begin{exmp}
Let $f$ be as in \cref{thm:main2} and assume $f_1$ has only zeros in the interior of~$\mathcal{F}$ and at infinity. Write $a_0(f)$ for the constant term at infinity of $f$, and $a(f_1)$ for the first non-zero Fourier coefficient of $f_1$. Then, in case $a_0(f)=0$ (as is the case when $f$ is the derivative of a modular form), or if $\sgn(a_0(f))=-\sgn(a(f_1))$, a direct evaluation of the result implies that 
\[ N_\lambda(f) = N_\lambda(f_1). \]
The situation alters slightly if $\sgn(a_0(f))=\sgn(a(f_1))$ (as is the case if $f=E_2$); in that case we find
\[ N_\lambda(f) \= \begin{cases}  
N_\lambda(f_1) & |\lambda|\in (\frac{1}{2},\infty) \\
N_\lambda(f_1)+1 & |\lambda|\in (0,\frac{1}{2}).
\end{cases}\]
\end{exmp}

\begin{exmp}\label{ex:intro2}
We return again to the example in the introduction. Applying Equations~\eqref{eq:N1/2} and~\eqref{eq:N0} in Theorem~\ref{thm:main2} and using the computations in Example~\ref{ex:intro}, we find
\begin{align}
    N_{(\frac{1}{2}, 1)}(f) \=  -1 \+ \left \lfloor \frac{36}{6} \right \rfloor \= 5.
\end{align}
Moreover, as $r(f_1)=1$, $m=0$, $f(\rho)<0$ and $s(f)=1$, we obtain
\begin{align}
    N_{[0,\frac{1}{2})}(f) \= -1 \+ \left \lceil \frac{36}{6} \right \rceil 
    \meno - \frac{1}{2} \meno \frac{-1}{2} \= 6.
\end{align}
\end{exmp}

\paragraph{Example in higher depth}
\begin{exmp}\label{ex:criticalE2}
Let $f=E_2^2-E_4$. Its zeros are the critical points of $E_2$. We have
$$h_j(\tau) = \tau + \frac{12}{2 \pi \ii} \frac{E_2(\tau) +(-1)^j\sqrt{E_4(\tau)}}{( E_2(\tau)^2-E_4(\tau) )},$$
and for $U=\{z\in \mathfrak{h}\mid \Im z>\frac{1}{2}\sqrt{3}\}$ we find 
$\sigma(C)=\sigma(T)=e$ (see \cref{cor:meromorphic} for the definition of~$\sigma$).
In particular, $h(-\frac{1}{2}+\ii t) \in -\frac{1}{2}+\ii \r$ for all $t\in \r$. Using the same ideas as in the proof of \cref{thm:5.3}, we find that the contribution of $\mathcal{C}$ equals $\frac{k-p}{6} -2 N_\infty(f)$ (with $k=4,p=2$). For the other contributions, we check that the proof goes through for both $E_2 \pm 
\sqrt{E_4}$ (which is not a \emph{holomorphic} quasimodular form). That is, take $h$ for $f_0=\pm \sqrt{E_4}$ and $f_1=1$. Then, the function $h$ for $E_2+\sqrt{E_4}=2+O(q)$ behaves as $(-,+)$ in \cref{tab:4}, whereas the function $h$ for $E_2-\sqrt{E_4}=-144q+O(q^2)$ behaves as $(-,-)$ in the same table. Hence,
\[ N_\lambda(f) +N_{-\frac{1}{\lambda}}(f)\= \begin{cases}
\frac{4-2}{6}+\frac{5}{6}-\frac{1}{6} \= 1 & |\lambda|<\frac{1}{2} \text{ or } |\lambda|>2 \\
\frac{4-2}{6}-\frac{1}{6}-\frac{1}{6} \= 0 & \frac{1}{2}<|\lambda|<2.\end{cases}\]

Observe that, in contrast to the case where the depth is 1, we no longer have that $|h_j(z)|=1$ for $|z|=1$. In particular,  $h_1(z)$ and $h_2(z)$ intersect the real line for $z\in \mathcal{C}$ in the  value $v$ and $\frac{1}{v}$ respectively, given by
\[ \frac{1}{v} = 0.180008\ldots, \qquad v=5.555295\ldots \]
As the value of $N_\lambda$ for positive $\lambda$ only changes at $\lambda=\frac{1}{2},\frac{1}{v},v$, and $N_\infty(f)\geq 1$ (because $\ii\infty$ is a zero of $f$), we conclude
\[ N_\lambda(f) \= \begin{cases}
1 & \frac{1}{v}<|\lambda|<\frac{1}{2} \text{ or } |\lambda|>v \\
0 & |\lambda|<\frac{1}{v} \text{ or } \frac{1}{2}<|\lambda|<v. \\
\end{cases}\]
\end{exmp}

\paragraph{Another result on the critical points of $E_2$}
Again, let $f=E_2^2-E_4$. Let 
\[\mathcal{F}_0(2) \defis \{z \in \mathfrak{h} \mid 0 \leq \Re z\leq 1 \text{ and } |z-\tfrac{1}{2}|\geq \tfrac{1}{2}\}\cup\{\ii \infty\}\]
be (the closure of) a fundamental domain for $\Gamma_0(2)$. In \cite{CL19} it is shown that
\begin{align}
\sum_{\tau \in \gamma\mathcal{F}_0(2)} \nu_\tau(f) = 1 \label{eq:criticalE2}
\end{align}
for all $\gamma \in \Gamma_0(2)$. In particular, the number of critical points of $E_2$ is constant in every $\gamma$-translate of $\mathcal{F}_0(2)$, but depends on $\lambda(\gamma)$ in every  $\gamma$-translate of $\mathcal{F}$ (see the previous example). Why are the zeros of a quasimodular form for $\sltwoz$ `better' distributed with respect to $\Gamma_0(2)$?

To get some more insight, we sketch how the proof of \cref{thm:5.3} can be adapted in order to give an alternative proof of~\eqref{eq:criticalE2}. Let
\[ N_\lambda^{(2)}(f) \defis \sum_{\tau \in \gamma\mathcal{F}_0(2)} \frac{\nu_\tau(f)}{e_{\tau,2}},\]
where $\lambda(\gamma)=\lambda$ and $e_{\tau,2}=2$ if $\tau$ is a $\gamma$-translate of $\frac{1}{2}+\frac{1}{2}\ii$ for $\gamma \in \Gamma_0(2)$, and $e_{\tau,2}=1$ else. Then, we claim
\[ N_{\lambda}^{(2)}(f) +  N_{\frac{\lambda-1}{2\lambda-1}}^{(2)}(f) \= 2\]
for all $\lambda$. Observe that $z\mapsto \frac{z-1}{2z-1}$ has the circle centered around $\frac{1}{2}$ with radius $\frac{1}{2}$ as its fixset. Now, the integral over the corresponding circular segment in the upper half plane yields a contribution of 
\[ \frac{k-p}{2}-2N_{\infty}^{(2)}(f)\]
with $k=4,p=2$ (namely, we integrate a function containing a $(k-p)$fold pole at $\lambda$ and $\frac{\lambda-1}{2\lambda-1}$ over $\frac{1}{2}$ of the circle). Moreover, the functions $h$ corresponding to $E_2\pm \sqrt{E_4}$ tend to $0$ and $1$ as $z$ tends to $0$ and $1$, and tend to $+\ii\infty$ for $\tau \to \ii\infty$ and $\tau\to 1+\ii \infty$. As exactly one of $\lambda$ and $\frac{\lambda-1}{2\lambda-1}$ lies between $0$ and $1$, we find in both cases that the contribution to the variation of the argument is $\frac{1}{2}$, which yields the claim. 

Next, let $U=\mathrm{int}\mathcal{F}_0(2)\backslash[\rho,\frac{1}{2}+i\infty)$. This is an open subset of $\mathfrak{h}$ invariant under $TC$ and $TST^2S$ (corresponding to $z\mapsto \frac{z-1}{2z-1}$). Moreover, $\sigma(CT)=(1\,2)$ and $\sigma(TST^2S)=(1\,2)$. As both leave the circle $\tfrac{1}{2}+\tfrac{1}{2}e^{\ii \theta}$ invariant, we conclude 
\[ |h_i(\tfrac{1}{2}+\tfrac{1}{2}e^{\ii \theta})-\tfrac{1}{2}|=\tfrac{1}{4}. \]
Hence, find that $N_\lambda^{(2)}(f)$ as a function of $\lambda$ can only change value at $\lambda=0,1$. Showing by other means that $N_
 \infty^{(2)}(f)=1$, one could conclude that
 \[N_\lambda^{(2)}(f)=1 \text{ for all $\lambda$}.\]

\appendix

\section{The zeros of \texorpdfstring{$\Re(\widehat{E}_2^n)$}{the real} and \texorpdfstring{$\Im(\widehat{E}_2^n)$}{imaginary part of powers of the quasimodular Eisenstein series of weight 2}}
We are interested in counting the number of zeros of~$\Re(\widehat{E}_2^n)$ and~$\Im(\widehat{E}_2^n)$ on $(\frac{\pi}{3}, \frac{2 \pi}{3})$ for $n > 0$. This is, for instance, motivated by the following example. 

\begin{exmp}
Write $f = f_0 + E_2^p$, where $f_0$ is a modular form of weight $2p>0$.
In this case, 
\begin{equation}
    \Im(\widehat{f}) = \Im(\widehat{E}_2^p).
\end{equation}
Using \cref{prop:rootsreimE2} below, the function $\widehat{f}$ crosses the imaginary axis exactly 
$
    p - 1 - 2 \lfloor \frac{p}{3}\rfloor
$
times. This means that the variation of the argument of $\widehat{f}$ along $\mathcal{C}$ is at most 
$
        \pi \bigl( p - 2 \bigl \lfloor \frac{p}{3}\bigr \rfloor \bigr).
$
Therefore, 
\begin{equation}
    N_\infty(f) \:\leq\: \frac{1}{2} \Bigl\lfloor\frac{2p}{6}\Bigr\rfloor + \frac{1}{2}\Bigl(p-2\Bigl\lfloor\frac{p}{3}\Bigr\rfloor\Bigr) \:\leq\: \frac{p+1}{3}.
\end{equation}
\end{exmp}

\begin{prop}\label{prop:rootsreimE2}
The functions $\Re(\widehat{E}_2^n)$ and~$\Im(\widehat{E}_2^n)$ admit 
\[n - 2 \Bigl\lfloor \frac{n}{3} + \frac{1}{2} \Bigr\rfloor, \qquad \text{resp.} \qquad n - 1 - 2 \Bigl \lfloor \frac{n}{3}\Bigr\rfloor  \]
zeros on $(\frac{\pi}{3}, \frac{2 \pi}{3})$ for $n \geq 0$.
\end{prop}
The proof almost immediately follows from the following result. 
\begin{lem}
For $n>0$, write $$(x+\ii)^n = R_n(x) + \ii \,Q_n(x),$$
where $R_n(x), Q_n(x)\in \mathbb{Z}[x]$. Then all roots of $R_n$ and $Q_n$ are real, of which 
\[n - 2 \Bigl\lfloor \frac{n}{3} + \frac{1}{2} \Bigr\rfloor, \qquad \text{resp.} \qquad n - 1 - 2 \Bigl \lfloor \frac{n}{3}\Bigr\rfloor  \]
lie in $(-\frac{1}{\sqrt{3}},\frac{1}{\sqrt{3}})$.
\end{lem}
\begin{proof}
For $x>0$, we may write
\begin{equation}
    (x+\ii)^n = (x^2+1)^{n/2} e^{\ii  \, n  \arctan(1/x) }.
\end{equation}
Therefore, for all $x$, we recognize 
\begin{align}
    R_n(x) &= (x^2+1)^{n/2} \, T_n( \cos( \arctan(1/x)))\\
                                &= (x^2+1)^{n/2} \, T_n \Bigl( \frac{x}{\sqrt{x^2+1}}\Bigr),
\end{align}
where $T_n$ is the $n$-th Chebyshev polynomial of the first kind, admitting the $n$ distinct real roots $\cos\bigl(\frac{(k+\frac{1}{2})\pi}{n}\bigr)$ for $k=0,\ldots,n-1$ in~$(-1,1)$. Hence, $R_n$ has $n$ distinct real roots, of which $n - 2 \bigl \lfloor \frac{n}{3} + \frac{1}{2} \bigr\rfloor$ are contained in $(-\tfrac{1}{\sqrt{3}},\tfrac{1}{\sqrt{3}})$.

Write $U_n$ for the $n$-th Chebyshev polynomial of the second kind. Then, for $n \geq 1$,
\begin{align}
    Q_n(x) &= (1+x^2)^{(n-1)/2} \, U_{n-1}(\cos(\arctan(1/x))) \\
                                 &= (1+x^2)^{(n-1)/2} \, U_{n-1}\Bigl( \frac{x}{\sqrt{1+x^2}}\Bigr),
\end{align}
from which one deduces that all roots of $Q_n$ are real and $n - 1 - 2 \bigl \lfloor \frac{n}{3}\bigr\rfloor$ lie in $(-\frac{1}{\sqrt{3}},\frac{1}{\sqrt{3}})$. Note that $Q_n$ admits a zero at $\pm \frac{1}{\sqrt{3}}$ for $n\equiv 0 \mod 3.$
\end{proof}
\begin{proof}[Proof of Proposition~\ref{prop:rootsreimE2}] We apply the previous lemma to $x = \frac{\pi}{3}\Re(\widehat{E}_2)$. Note that by~\eqref{eq:imE2hat}, we have $\Im(\widehat{E}_2) = \frac{3}{\pi}$. Hence,
\begin{align}
    \frac{\pi^n}{3^n}\Re (\widehat{E}_2^n ) &= \Re\bigl(\bigl(\frac{\pi}{3}\Re\bigl(\widehat{E}_2\bigr) + \ii\bigr)^n \bigr) 
    =R_n\bigl(\frac{\pi}{3}\Re\bigl(\widehat{E}_2\bigr)\bigr)
\end{align}
and
\begin{align}
       \frac{\pi^n}{3^n} \Im ( \widehat{E}_2^n ) = Q_n\bigl(\frac{\pi}{3}\Re\bigl(\widehat{E}_2\bigr)\bigr).
\end{align}

Observe that~$\Re(\widehat{E}_2)$ is a strictly decreasing function on $(\frac{\pi}{3}, \frac{2 \pi}{3})$ with a unique zero at $\theta = \frac{\pi}{2}$.
As on the boundary we have 
$$\frac{\pi}{3}\Re(\widehat{E}_2)\Bigl(\frac{\pi}{3}\Bigr) = -\frac{\pi}{3}\Re(\widehat{E}_2)\Bigl(\frac{2 \pi}{3}\Bigr) = \frac{1}{\sqrt{3}},$$
the proposition follows from the lemma above.
\end{proof}

\end{document}